\documentclass[preprint,12pt]{elsarticle}

\usepackage{lineno,hyperref}
\modulolinenumbers[5]

\usepackage{amssymb}
\usepackage{algorithmic}
\usepackage{algorithm}
\usepackage{amssymb}
\usepackage{graphicx}
\usepackage{amsmath,amsfonts,amssymb}
\usepackage{epsfig,makecell,float}
\usepackage{setspace,mathrsfs}
\usepackage{tocloft}
\usepackage{textcomp}
\usepackage{multirow,indentfirst,times,color}
\usepackage[caption=false,font=footnotesize]{subfig}
\usepackage{booktabs}
\usepackage{threeparttable}
\usepackage{amsthm}
\usepackage{extarrows}
\usepackage{bm}

\renewcommand{\arraystretch}{1.5}
\usepackage[titletoc]{appendix}
\usepackage{epstopdf}
\usepackage{natbib}
\biboptions{numbers,sort&compress}
\newtheorem{thm}{Theorem}
\newtheorem{defn}{Definition}
\newtheorem{remark}{Remark}

\DeclareMathOperator*{\argmin}{arg\,min}
\DeclareMathOperator*{\argmax}{arg\,max}

\begin{document}

\begin{frontmatter}

\title{Graph Fractional Fourier Transform: A Unified and Efficient Sampling Theory}

\author{Yu Zhang$^{a,b}$}
\author{Jia-Yin Peng$^{a,b}$}
\author{Bing-Zhao Li$^{a,b}$\corref{mycorrespondingauthor}}
\cortext[mycorrespondingauthor]{Corresponding author}\ead{li\_bingzhao@bit.edu.cn}

\address{$^{a}$School of Mathematics and Statistics, Beijing Institute of Technology, Beijing 100081, China}
\address{$^{b}$Beijing Key Laboratory on MCAACI, Beijing Institute of Technology, Beijing 100081, China}


\begin{abstract}
The graph Fourier transform (GFT) is a fundamental tool in graph signal processing and has recently been extended to the graph fractional Fourier transform (GFRFT). Existing sampling methods in the GFRFT domain are primarily designed to minimize error, whereas a wider range of alternative sampling strategies should be admitted. In this paper, a unified and efficient GFRFT sampling theory is proposed. First, a new definition of graph fractional bandlimited signals is introduced, with the corresponding graph fractional sampling and perfect reconstruction theorem, as well as the associated graph fractional localization operator. Next, several GFRFT sampling strategies are developed based on different criteria, including maximum cutoff frequency, minimum error, and maximum localized basis, along with the corresponding representations of their localization operators. Then, by exploiting a localization operator that jointly considers vertex and spectral localization, a fast sampling set selection method in the GFRFT domain is proposed. Finally, numerical experiments investigate the reconstruction errors and execution time of the proposed sampling methods and evaluate their performance in applications, demonstrating the effectiveness of the unified GFRFT sampling theory and its advantages over GFT methods.
\end{abstract}


\end{frontmatter}

\section{Introduction}
\label{Intro}
\subsection{Motivation}
To process data defined on irregular structures, classical signal processing techniques have been extended to non-Euclidean graph topologies, giving rise to the framework of graph signal processing (GSP) \cite{Goverview}, which has been applied in transportation networks \cite{GFTlaplace,GFTadjacency,Ghistory}, social networks \cite{Social1,Social2}, biomedical analysis \cite{Biomedical}, and machine learning \cite{MachineLearning}. GSP generalizes many fundamental concepts and tools from classical signal processing to graphs, such as graph transforms \cite{Gtransform, GFRFT_unified,GLCT,Gvertex}, vertex-frequency analysis \cite{GLCT,Gvertex,Gfrequency}, sampling and interpolation \cite{GFTsampling,GFTuncertainty,GFTefficient,GFTSSS,GFTdualizing,GFTgenersamp,GFTinterpolation}, filtering \cite{GFTfilters}, and fast algorithms \cite{GFTfast}.

Two main approaches are commonly used in GSP. One is based on the graph Laplacian matrix and originates from spectral graph theory \cite{GFTlaplace}, while the other relies on the adjacency matrix and stems from algebraic signal processing, commonly referred to as discrete signal processing on graphs \cite{GFTadjacency}. These matrices are collectively known as graph shift operators. The core tool of GSP is the graph Fourier transform (GFT), which is defined via the eigendecomposition of the graph shift operator and represents graph signals in the eigenvector basis of the graph. The GFT can be regarded as a direct generalization of the discrete Fourier transform to graphs \cite{Gtransform}. For many real-world graph signals whose energy is concentrated in the low-frequency components of the GFT domain, a variety of signal processing tasks can be effectively performed.

Despite its well-established theoretical foundation, the GFT is inherently limited in its ability to characterize the transition from the vertex domain to the spectral domain and to process graph signals exhibiting chirp-like behaviors \cite{GFRFT,GFRFTspectral}. These limitations result in insufficient degrees of freedom and reduced flexibility. To address this issue, early studies \cite{GFRFT,GFRFTspectral, GFRFTdirected} innovatively extended the fractional Fourier transform \cite{FRFT1,FRFT2,DFRFT} to graphs, leading to the graph fractional Fourier transform (GFRFT).

\subsection{Related Work}
More recent GFRFT works have established rigorous definitions, including power-based formulations for arbitrary graphs and operator constructions based on hyper-differential equations \cite{GFRFT_unified}. These advances enable the efficient computation of the GFRFT and its inverse using Sylvester equation solvers and allow the fractional order to be treated as a differentiable and learnable parameter in data-driven models. As a result, the GFRFT has been successfully applied to vertex-frequency analysis \cite{GFRFT_unified,WGFRFT}, classification \cite{JFRFT,GFRFTsampling,GLCTsampling}, sampling \cite{GFRFTsampling,GLCTsampling,HGFRFT}, and filters \cite{HGFRFT,GFRFTfiltering,JFRFTfilter,JFRFTwiener}. Building on these developments, the GFRFT also provides a powerful foundation for designing more efficient sampling strategies in GSP. 

For large-scale or complex graph data \cite{Goverview,GFTlaplace,GFTadjacency,Ghistory}, signal acquisition and processing require substantial storage and computational resources. Sampling plays a critical role in reducing data redundancy while preserving essential signal characteristics \cite{Gsampling}, motivating the need for methods that can capture the most informative components of graph signals. Most existing sampling approaches focus on smooth or bandlimited signals, whose energy is concentrated on a subset of eigenvectors of the graph shift operator \cite{GFTsampling,GFTuncertainty,GFTefficient,GFTSSS,GFTdualizing,GFTgenersamp,GFTinterpolation}. As a result, sampling is typically formulated by selecting the graph spectral components associated with the largest GFT coefficients \cite{Gfrequency,GFTsampling,GFTuncertainty,GFTefficient,GFTSSS,GFTdualizing,GFTgenersamp}.

Previous studies have investigated optimal sampling strategies for signals with known frequency support \cite{GFTsampling}, as well as the relationship between signal localization on graphs and spectral spreading \cite{GFTuncertainty}. Noise-robust reconstruction methods and blue-noise sampling schemes have also been proposed \cite{GFTblue}. In addition, scalable and parallel sampling and reconstruction frameworks \cite{GFTParallel}, as well as non-Bayesian and variational Bayesian estimators based on the Cramér-Rao bound \cite{GFTNon-Bayesian,GFTBayes}, have been developed. Sampling and reconstruction of joint time-vertex signals were studied in \cite{JFTsampling,JFTdirectedsamp}, where most bandlimited sampling methods employ greedy algorithms to select optimal sampling sets.

\subsection{Contributions}
The sampling theory in the GFRFT domain was first introduced in \cite{GFRFTsampling}, where optimal sampling sets were obtained by minimizing error. Subsequently, generalized sampling methods exploiting GFRFT-based graph signal priors and sampling approaches based on the joint time-vertex fractional Fourier transform were proposed \cite{GFRFTgenersamp,JFRFTsampling}. However, existing GFRFT sampling methods rely on a single optimization criterion and do not explore a broader class of optimal design strategies. Moreover, the use of greedy and iterative procedures often leads to high computational complexity.

To address these limitations, this paper proposes a unified GFRFT sampling framework that incorporates multiple sampling strategies for signal reconstruction. Moreover, inspired by the pioneering approach in \cite{GFTSSS}, a fast sampling set selection method is integrated into the proposed framework to significantly reduce computational complexity and execution time. 
Extensive numerical simulations and real-world experiments demonstrate the superiority of the proposed methods in terms of reconstruction accuracy and computational efficiency, providing new insights into efficient GSP strategies.

The main contributions are summarized as follows:
\begin{itemize}
	\item A unified and efficient GFRFT sampling framework is developed, including graph fractional bandlimited signal, sampling and perfect reconstruction conditions, and localization theory.
	\item Several GFRFT sampling strategies are proposed under different optimization criteria, such as maximum cutoff frequency, minimum reconstruction error, and maximum localization, with their localization operators.
	\item A fast sampling set selection algorithm is introduced based on a localization operator that jointly captures vertex and spectral domain concentration.
	\item The effectiveness and computational efficiency of the proposed GFRFT sampling methods are validated through extensive experiments.
\end{itemize}
The remainder of this paper is organized as follows. 
Section~\ref{Preliminaries} reviews the background on GSP, the GFRFT, graph signal sampling, and localization operators. 
Section~\ref{GFRFTSampling} develops the theoretical framework of graph fractional sampling, including bandlimitedness, sampling, and recovery. 
Section~\ref{OptimalSampling} investigates optimal fractional sampling strategies based on different design criteria. 
Section~\ref{FastSampling} presents a fast GFRFT sampling set selection method. 
Section~\ref{Experiments} reports experimental results that demonstrate the advantages of the proposed methods in terms of reconstruction accuracy and computational efficiency. 
Finally, Section~\ref{Conclusion} concludes the paper.

\section{Background}
\label{Preliminaries}
In this section, we first briefly review some basic concepts in GSP including graph setting and GFRFT. Next, we introduce classical sampling theory and localization operator.

\subsection{Basic Concepts in GSP}
In GSP \cite{Goverview,GFTlaplace, GFTadjacency, Ghistory}, signals with high-dimensional topological structure can be represented by an undirected weighted graph $\mathcal{G} = (\mathcal{V}, \mathcal{E}, \mathbf{W})$. The graph consists of a finite set of vertices $\mathcal{V} = \{v_0, \ldots, v_{N-1}\}$, where $N = |\mathcal{V}|$ denotes the number of nodes. The set $\mathcal{E}$ represents the edges, and $\mathbf{W}$ is the weighted adjacency matrix with entries $w_{m,n}$. The unnormalized graph Laplacian is defined as $\mathbf{L} = \mathbf{D}^{\mathrm{deg}} - \mathbf{W}$, where $\mathbf{D}^{\mathrm{deg}} = \mathrm{diag}\big(d^{\mathrm{deg}}_0, \ldots, d^{\mathrm{deg}}_{N-1}\big)$ is the degree matrix whose $n$-th diagonal element is given by $k_n := \sum_{m=0}^{N-1} w_{m,n}$. Adjacency and Laplacian matrices are commonly used as symmetric graph shift operators $\mathbf{S}$, which admit the eigendecomposition
\begin{equation}
	\mathbf{S} = \mathbf{U}\boldsymbol{\Delta}\mathbf{U}^{\mathrm{H}},
	\label{shift}
\end{equation}
where $\bm{\Delta}=\mathrm{diag}(\delta_0,\ldots,\delta_{N-1})$ with $0\le\delta_0\le\cdots\le\delta_{N-1}=\delta_{\max}$, and $\mathbf{U}=[\bm{u}_0,\bm{u}_1,\ldots,\bm{u}_{N-1}]$ contains orthogonal eigenvectors.

A signal defined on the graph $\mathcal{G}$ is a mapping from the vertex set $\mathcal{V}$ to the complex field $\mathbb{C}$. The graph signal can be expressed as a vector $\bm{x} = [x_0, \ldots, x_{N-1}]^{\top} \in \mathbb{C}^{N}$. The GFT \cite{GFTadjacency} of $\bm{x}$ is defined as
\begin{equation}
	\hat{\bm{x}} = \mathbf{F}\bm{x} = \mathbf{U}^{\mathrm{H}} \bm{x}. \label{GFT}
\end{equation}
Projecting $\bm{x}$ onto the eigenspace of $\mathbf{S}$ exploits the graph topology, as these eigenvectors encode spectral clustering properties \cite{SpectralGraph}. Consequently, the GFT in \eqref{GFT} emphasizes signal components that are smooth within clusters and vary across clusters. The inverse GFT is given by $\bm{x} = \mathbf{U}\hat{\bm{x}}$, where $\hat{\bm{x}} = [\hat{x}_0, \ldots, \hat{x}_{N-1}]^{\top}$ denotes the frequency coefficients in the transform domain.

\begin{remark}
For directed graphs, the shift operator is generally non-symmetric, and a Jordan decomposition is required to diagonalize $\mathbf{S}$ \cite{Gtransform}. In contrast, when $\mathbf{S} = \mathbf{S}^{\mathrm{H}}$, the GFT is unitarily diagonalizable and reduces to $\mathbf{F} = \mathbf{U}^{-1} = \mathbf{U}^{\mathrm{H}}$.
\end{remark}

\subsection{Graph Fractional Fourier Transform}
The GFRFT is defined in analogy with the eigendecomposition formulation of the discrete fractional Fourier transform. By extending the definition in \cite{GFRFT_unified,GFRFT} to accommodate a graph shift operator $\mathbf{S}$, the $\alpha$-th order GFRFT of a graph signal $\bm{x}\in\mathbb{C}^N$ is given by
\begin{equation}
	\hat{\bm{x}} = \mathbf{F}^{\alpha}\bm{x}
	= \mathbf{Q}\mathbf{\Lambda}^{\alpha}\mathbf{Q}^{\mathrm{H}}\bm{x},
	\quad \alpha \in \mathbb{R}, \label{GFRFT}
\end{equation}
and the inverse GFRFT is
\begin{equation}
	\bm{x} = \mathbf{F}^{-\alpha}\hat{\bm{x}} = \mathbf{Q}\mathbf{\Lambda}^{-\alpha}\mathbf{Q}^{\mathrm{H}}\hat{\bm{x}}, \label{IGFRFT}
\end{equation}
where matrix $\mathbf{Q}$ and $\mathbf{\Lambda}$ are given by the spectral decomposition of GFT matrix 
\[
\mathbf{U}^{\mathrm{H}} = \mathbf{Q}\mathbf{\Lambda} \mathbf{Q}^{\mathrm{H}}.
\] 
The term $\mathbf{\Lambda}^{\alpha}$ is computed by applying the power a to every entry in the diagonal matrix $\mathbf{\Lambda}=\mathrm{diag}(\lambda_0, \lambda_1, \ldots, \lambda_{N-1})$. As shown in \cite{GFRFT}, the operator satisfies index additivity, reduces to the identity transform when $\alpha = 0$, and recovers the conventional GFT when $\alpha = 1$.

\subsection{Sampling Theory of Graph Signals}
Sampling is a fundamental problem in GSP \cite{Gsampling}. The objective is to determine conditions under which a bandlimited or approximately bandlimited graph signal can be recovered from a subset of its samples, and to design effective sampling and reconstruction strategies \cite{GFTsampling,GFTuncertainty,GFTefficient}. A bandlimited graph signal can be written as
\begin{equation}
	\bm{x} = \mathbf{U}\bm{s},
\end{equation}
where $\mathbf{U}$ is defined in Eq.~\eqref{shift}, and $\bm{s}$ denotes the sparse graph spectral coefficient vector. Given a vertex subset $\mathcal{S} \subseteq  \mathcal{V}$, the vertex-limiting operator is defined as the diagonal matrix
\begin{equation}
	\mathbf{D}_{\mathcal{S}} = \mathrm{diag}(\bm{1}_{\mathcal{S}}), \label{DS}
\end{equation}
where the indicator vector $\bm{1}_{\mathcal{S}}$ has its $i$-th entry equal to $1$ if $i \in \mathcal{S}$ and $0$ otherwise. For a sampling set $\mathcal{S}$, the sampled signal is
\begin{equation}
	\bm{x}_{\mathcal{S}} 
	= \mathbf{D}_{\mathcal{S}}\bm{x}
	= \mathbf{D}_{\mathcal{S}}\mathbf{U}\bm{s}.
	\label{sampledsignal}
\end{equation}
Recovering $\bm{x}$ from its samples amounts to solving \eqref{sampledsignal} by exploiting the sparsity of $\bm{s}$, for which both iterative and non-iterative reconstruction methods have been studied \cite{GFTsampling,GFTreconstruction}.

It is important to note that the choice of the sampling set $\mathcal{S}$ plays a critical role, because it directly influences the conditioning of~\eqref{sampledsignal}. Therefore, designing strategies for optimizing the sampling set is essential for reliable reconstruction.

\subsection{Localization Operators}
The filtering operation in the graph spectral domain is given by
\begin{equation}
	\bm{x} = \mathbf{U}\, h(\mathbf{\Delta})\, \mathbf{U}^{\mathrm{H}} \bm{y}, \label{x=y}
\end{equation}
where $\mathbf{\Delta}$ is defined in Eq.~\eqref{shift} with spectral kernel $h(\delta_i)$, and $\bm{y}$ and $\bm{x}$ are the input and output signals, respectively. 

The $n$-th component of the localization operator centered at vertex $i$ is defined in \cite{Gglobal} as
\begin{equation}
	T_{h,i}(n)
	= \sqrt{N}\sum_{\ell=0}^{N-1}
	h(\delta_{\ell})\, u_{\ell}^{\mathrm{H}}(i)\, u_{\ell}(n). \label{Thi}
\end{equation}
Stacking the localization operators as rows gives
\begin{equation}
	\mathbf{T}
	= \big[\, \bm{T}_{h,0}\; \bm{T}_{h,1}\; \cdots\; \bm{T}_{h,N-1} \big]
	= \sqrt{N}\, \mathbf{U} h(\bm{\Delta}) \mathbf{U}^{\mathrm{H}}.  \label{T}
\end{equation}

\section{Graph Fractional Sampling Theory}
\label{GFRFTSampling}
This section builds on the sampling framework in \cite{GFTsampling,GFRFTsampling} and develops an efficient and unified GFRFT-based sampling theory. The proposed approach exploits orthogonal bases and spectral-domain sampling to reconstruct graph bandlimited signals from undersampled observations. 

\subsection{Graph Fractional Bandlimited Signals}
\label{sec3.1}
Graph fractional bandlimited signals have been characterized through several complementary perspectives in prior studies \cite{GFRFTsampling}. Here, we provide a unified formulation by combining vertex bandlimiting operators with fractional spectral bandlimiting operators, thereby establishing a coherent framework for fractional bandlimitation and sampling.

We begin with the vertex domain. For a vertex subset $\mathcal{S} \subseteq \mathcal{V}$, the vertex limiting operator $\mathbf{D}_{\mathcal{S}}$ is defined in Eq.~\eqref{DS}. This operator is a square matrix whose function is to zero out all entries associated with vertices outside the subset $\mathcal{S}$.

In the dual domain, given the GFRFT matrix $\mathbf{F}^{\alpha}$ (see Eq.~\eqref{GFRFT}) and a frequency-index subset $\mathcal{F} \subseteq \hat{\mathcal{G}} = \{1,\dots,N\}$, the spectral limiting operator is defined as
\begin{equation}
	\mathbf{B}_{\mathcal{F}}^{\alpha}
	= \mathbf{F}^{-\alpha} \bm{\Sigma}_{\mathcal{F}} \mathbf{F}^{\alpha},
\end{equation}
where $\bm{\Sigma}_{\mathcal{F}} = \mathrm{diag}(\mathbf{1}_{\mathcal{F}})$, and $\mathbf{1}_{\mathcal{F}}$ denotes the indicator vector of $\mathcal{F}$, analogously to $\mathbf{1}_{\mathcal{S}}$.
The operator $\mathbf{B}_{\mathcal{F}}^{\alpha}$ projects a graph signal onto the subspace spanned by the columns of $\mathbf{F}^{\alpha}$ indexed by $\mathcal{F}$. Both $\mathbf{D}_{\mathcal{S}}$ and $\mathbf{B}_{\mathcal{F}}^{\alpha}$ are Hermitian and idempotent, and hence act as orthogonal projection operators.

For notational simplicity, we omit the subscripts associated with the sets. Let $\mathcal{D}$ denote the set of vertex-limited signals satisfying $\mathbf{D}\bm{x} = \bm{x}$, and let $\mathcal{B}^{\alpha}$ denote the set of fractional bandlimited signals satisfying $\mathbf{B}^{\alpha}\bm{x} = \bm{x}$. The complement of $\mathcal{S}$, denoted $\mathcal{S}^c$, satisfies $\mathcal{V} = \mathcal{S} \cup \mathcal{S}^c$, with associated projection $\overline{\mathbf{D}} = \mathbf{I} - \mathbf{D}$. Similarly, the complement of $\mathcal{F}$ induces the projection $\overline{\mathbf{B}}^{\alpha} =   \mathbf{F}^{-\alpha}(\mathbf{I}-\bm{\Sigma})\mathbf{F}^{\alpha} = \mathbf{I} - \mathbf{B}^{\alpha}$. These definitions lead to the following fundamental characterization of graph fractional bandlimited signals.

\begin{thm}
	A graph signal $\bm{x} \in \mathbb{C}^{N}$ is perfectly localized, meaning that it is simultaneously bandlimited over the pair $(\mathcal{S},\mathcal{F})$ so that $\bm{x} \in \mathcal{D} \cap \mathcal{B}^{\alpha}$, if and only if  \label{thm1}
	\begin{equation}
		\|\mathbf{B}^{\alpha}\mathbf{D}\|_2
		=
		\|\mathbf{D}\mathbf{B}^{\alpha}\|_2
		= 1.
		\label{eq:norm1}
	\end{equation}
\end{thm}
\begin{proof}
	The proof is provided in Appendix~\ref{AA}.
\end{proof}

Moreover, the vector $\bm{x}$ corresponds to an eigenvector associated with the unit eigenvalue. It is also worth noting that when $\alpha = 1$, the fractional bandlimited theorem reduces to the classical graph bandlimited case.

\subsection{Graph Fractional Sampling and Perfect Recovery}
\label{sec3.2}
Given a graph signal $\bm{x} \in \mathcal{B}^{\alpha}$ defined on the vertex set, the sampled signal $\bm{x}_{\mathcal{S}} \in \mathcal{D}$, obtained by retaining the entries of $\bm{x}$ on the sampling subset $\mathcal{S} \subseteq \mathcal{V}$ and setting all others to zero, as in Eq.~\eqref{sampledsignal}. A reconstruction operator $\mathbf{R}$ is then used to recover the original signal from its samples,
\begin{equation}
	\bm{x}_{\mathcal{R}} = \mathbf{R}\bm{x}_{\mathcal{S}}
	= \mathbf{R}\mathbf{D}\bm{x} \in \mathbb{C}^{N}, \label{xR=RDx}
\end{equation}
where $\bm{x}_{\mathcal{R}}$ denotes the exact or approximate reconstruction, as illustrated in Fig.~\ref{fig01}.
\begin{figure}[t]
	\begin{center}
		\begin{minipage}[t]{0.4\linewidth}
			\centering
			\includegraphics[width=\linewidth]{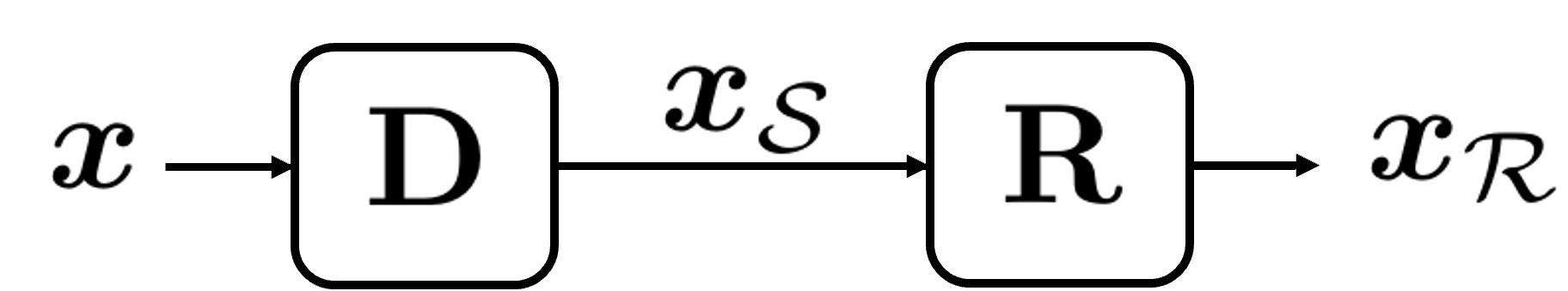}
		\end{minipage}
	\end{center}
	\caption{Sampling and reconstruction of a graph signal.}
	\label{fig01}
\end{figure}

It is immediate that perfect recovery is achieved for all signals if and only if $\mathbf{RD}$ equals the identity matrix. However, this is generally impossible because
\[
\mathrm{rank}(\mathbf{RD}) \leq \mathrm{rank}(\mathbf{D}) = |\mathcal{S}| < N.
\]
Nevertheless, perfect reconstruction is achievable for certain classes of signals with specific structural properties.

In particular, if a graph signal satisfies Theorem~\ref{thm1}, then perfect recovery from its samples is possible. Such signals are known as graph fractional bandlimited signals. An equivalent characterization is that a signal is graph fractional bandlimited if its GFRFT coefficients satisfy $\hat{\bm{x}}_k = 0$ for all $k \geq |\mathcal{F}|$. The smallest such index $|\mathcal{F}|$ is termed the bandwidth of $\bm{x}$. 

Following~\cite{GFRFTsampling}, the bandwidth in the GFRFT domain is defined as the number of nonzero coefficients in the fractional Fourier representation, reflecting the support of signal in the fractional spectral domain.

\begin{defn}
A signal $\bm{x} \in \mathbb{C}^{N}$ is said to be bandwidth-limited if the number of nonzero entries in its graph fractional spectrum $\hat{\bm{x}}$ is less than $|\mathcal{F}| < N$.
\end{defn}

Let $\mathcal{S}$ be an appropriate sampling subset. There exists a sampling set of size $|\mathcal{S}| \geq |\mathcal{F}|$ such that $\bm{x}$ can be perfectly recovered from its samples. The key issue is to determine the condition for perfect recovery from $\bm{x}_{\mathcal{S}}$. The following theorem states the result.

\begin{thm}
Given the samples $\bm{x}_{\mathcal{S}} = \mathbf{D}\bm{x}$, any $\bm{x} \in \mathcal{B}^{\alpha}$ can be perfectly recovered from its samples if and only if \label{thm2}
\[
\| \mathbf{B}^{\alpha}\overline{\mathbf{D}} \|_2 < 1,
\]
which means the operator $\mathbf{B}^{\alpha}\overline{\mathbf{D}}$ has no eigenvector that is fully supported on $\mathcal{S}^c$ while remaining graph fractional bandlimited on $\mathcal{F}$.
\end{thm}
\begin{proof}
	The proof is provided in Appendix~\ref{AB}.
\end{proof}

Based on Theorem~\ref{thm2} and Eq.~\eqref{xR=RDx}, we obtain a reconstruction scheme for recovering the original signal from its samples. For a sampling operator $\mathbf{D}$, the reconstruction operator $\mathbf{R}$ spans the $|\mathcal{F}|$-dimensional fractional bandlimited subspace induced by $\mathbf{F}^{\alpha}$, consistent with $\bm{x}\in\mathcal{B}^{\alpha}$ being graph fractional bandlimited. If $\mathbf{R}\mathbf{D}$ acts as a projection operator, then every $\bm{x}\in\mathcal{B}^{\alpha}$ satisfies $\bm{x}=\mathbf{R}\mathbf{D}\bm{x}
=\mathbf{R}(\mathbf{I}-\overline{\mathbf{D}}\mathbf{B}^{\alpha})\bm{x}$, which leads to the reconstruction formula
\[
\bm{x}_{\mathcal{R}}
=\mathbf{R}\bm{x}_{\mathcal{S}}
=(\mathbf{I}-\overline{\mathbf{D}}\mathbf{B}^{\alpha})^{-1}\bm{x}_{\mathcal{S}}.
\]

Before introducing a specific reconstruction operator, observe that for any $\bm{x}\in\mathcal{B}^{\alpha}$,
\[
\left(\mathbf{I}-\overline{\mathbf{D}}\mathbf{B}^{\alpha}\right)\bm{x}
=(\mathbf{I}-\overline{\mathbf{D}})\bm{x}
=\mathbf{D}\bm{x}
=\mathbf{D}\mathbf{B}^{\alpha}\bm{x}.
\]
Thus, $\mathbf{I}-\overline{\mathbf{D}}\mathbf{B}^{\alpha}$ is equivalent to $\mathbf{D}\mathbf{B}^{\alpha}$ when acting on bandlimited signals, and consequently $\mathbf{R}=(\mathbf{D}\mathbf{B}^{\alpha})^{-1}$. 
The operator $\mathbf{D}\mathbf{B}^{\alpha}$ is invertible if and only if
\[
\mathrm{rank}(\mathbf{D}\mathbf{B}^{\alpha})=\mathrm{rank}(\mathbf{B}^{\alpha}),
\]
which is precisely the condition in Theorem~\ref{thm2}. In this case, the singular vectors associated with the nonzero singular values of $\mathbf{D}\mathbf{B}^{\alpha}$ form a basis for the range of $\mathbf{B}^{\alpha}$. For the general possibly rank-deficient case, we employ the pseudo-inverse and define $\mathbf{R} = (\mathbf{D}\mathbf{B}^{\alpha})^{\dagger}$.

Therefore, the reconstruction operator can be expressed as
\begin{equation}
	\mathbf{R}
	=(\mathbf{D}\mathbf{B}^{\alpha})^{\dagger}
	=(\mathbf{D}\mathbf{F}^{-\alpha}\mathbf{\Sigma}\mathbf{F}^{\alpha})^{\dagger}
	=\mathbf{F}^{-\alpha}_{\mathcal{V}\mathcal{F}}
	\left(\mathbf{F}^{-\alpha}_{\mathcal{S}\mathcal{F}}\right)^{\dagger},
	\label{R}
\end{equation}
where $\mathbf{F}^{-\alpha}_{\mathcal{V}\mathcal{F}}$ denotes the submatrix of $\mathbf{F}^{-\alpha}$ containing all $N$ rows and the columns indexed by the frequency support $\mathcal{F}$, while $\mathbf{F}^{-\alpha}_{\mathcal{S}\mathcal{F}}$ selects the rows indexed by $\mathcal{S}$ and the same set of columns $\mathcal{F}$.

\subsection{Graph Fractional Localization Reconstruction and Error}
This part builds upon the localized operator in Eq.~\eqref{T} and develops a graph fractional localization framework that unifies sampling, reconstruction, and error characterization under a localized formulation.

Consider a polynomial graph filter defined in the vertex domain, where $h(\delta_i)$ denotes the spectral kernel and $\delta_i$ is the $i$-th eigenvalue of the graph shift operator $\mathbf{S}$. We can define the graph fractional shift operator
\begin{equation}
	\mathbf{S}^{\alpha} := \mathbf{F}^{-\alpha} \mathbf{\Delta}^{\alpha}\, \mathbf{F}^{\alpha}. \label{Salpha}
\end{equation}
Similar to Eq.~\eqref{x=y}, the GFRFT-based filtering process is expressed as
\[
\bm{x} = h(\mathbf{S}^{\alpha})\bm{y} := \mathbf{F}^{-\alpha} h(\mathbf{\Delta}^{\alpha}) \mathbf{F}^{\alpha} \bm{y},
\]
where $\mathbf{\Delta}^{\alpha} = \mathrm{diag}(\delta^{\alpha}_0, \dots, \delta^{\alpha}_{N-1})$ collects the spectral components in the vertex domain. Following \cite{Gglobal} and Eq.~\eqref{Thi}, the fractional localization operator is given by $T^{\alpha}_{h,i}(n) = \sum_{\ell=0}^{N-1} 
h(\delta^{\alpha}_\ell)\, (F^{\alpha}_\ell)^{\mathrm{H}}(i)\, F^{\alpha}_\ell(n)$, and $h(\delta^{\alpha}_\ell)$ typically represents a low-pass response, and $\bm{F}^{\alpha}_\ell = \left[ F^{\alpha}_{\ell}(0), \dots, F^{\alpha}_{\ell}(N-1)\right] $ denotes the $\ell$-th column of the fractional graph Fourier matrix $\mathbf{F}^{\alpha}$. The matrix representation of the localized operator is
\begin{equation}
	\mathbf{T}^{\alpha} = \big[\, \bm{T}^{\alpha}_{h,0}\; \bm{T}^{\alpha}_{h,1}\; \cdots\; \bm{T}^{\alpha}_{h,N-1} \big]:=
	\mathbf{F}^{-\alpha} h(\mathbf{\Delta}^{\alpha}) \mathbf{F}^{\alpha},  \label{Talpha}  
\end{equation}
here the selection of filter $h$ requires that $\mathbf{T}^{\alpha}$ be Hermitian. For the fractional localization operator, signal reconstruction has an equivalent alternative \cite{GFTSSS}, achieved through a weighted linear combination of the columns of $\mathbf{T}^{\alpha}$, as formalized below.

\begin{thm}
Let $\bm{x}\in\mathcal{B}^{\alpha}$ be a graph fractional bandlimited signal with bandwidth $|\mathcal{F}|$. If $|\mathcal{S}| \ge |\mathcal{F}|$ and the samples $\bm{x}_{\mathcal{S}}$ collected on $\mathcal{S}$ permit perfect recovery, then $\bm{x}$ can be exactly reconstructed by \label{thm3}
\begin{equation}
	\bm{x}_{\mathcal{R}}
	= \mathbf{T}^{\alpha}_{\mathcal{V}\mathcal{S}}
	\left( \mathbf{T}^{\alpha}_{\mathcal{S}} \right)^{\dag}
	\bm{x}_{\mathcal{S}},
	\label{xR=TTxS}
\end{equation}
where $\mathcal{V}$ denotes the full vertex index set and $\mathcal{S}\subseteq\mathcal{V}$ the sampling set.
\end{thm}
\begin{proof}
	The proof is provided in Appendix~\ref{AC}.
\end{proof}

Using Eqs.~\eqref{xR=RDx}, \eqref{R}, and \eqref{xR=TTxS}, the recovery operator $\mathbf{R}$ can be approximated as
\begin{equation}
	\mathbf{R}
	\approx
	\mathbf{T}^{\alpha}_{\mathcal{V}\mathcal{S}}
	\left( \mathbf{T}^{\alpha}_{\mathcal{S}} \right)^{\dag}.
\end{equation}

When the samples are contaminated by Gaussian noise $\bm{\xi}$, the reconstruction error is obtained directly from Eq.~\eqref{xR=TTxS}. The error is
\begin{equation}
	\bm{e}
	= \mathbf{R}(\bm{x}_{\mathcal{S}} + \bm{\xi}) - \bm{x}
	= \mathbf{R}\bm{\xi}
	= \mathbf{T}^{\alpha}_{\mathcal{V}\mathcal{S}}
	\left( \mathbf{T}^{\alpha}_{\mathcal{S}} \right)^{\dag}
	\bm{\xi}.  \label{e=TTxi}
\end{equation}
Equivalently,
\[
\begin{aligned}\bm{e}
	=&\mathbf{F}^{-\alpha}
	h(\mathbf{\Delta}^{\alpha})
	\mathbf{F}^{\alpha}_{\mathcal{S}\mathcal{V}}
	\left(
	\mathbf{F}^{-\alpha}_{\mathcal{S}\mathcal{V}}
	h(\mathbf{\Delta}^{\alpha})
	\mathbf{F}^{\alpha}_{\mathcal{S}\mathcal{V}}
	\right)^{\dag}
	\bm{\xi}\\ 
	=& \left( \mathbf{T}^{\alpha} \right)^{1/2}
	\big(
	( \mathbf{T}^{\alpha} )^{1/2}_{\mathcal{S}\mathcal{V}}
	\big)^{\dag}
	\bm{\xi}.\end{aligned} 
\]

The corresponding error covariance matrix is
\[
\mathbf{E}
=
\left( \mathbf{T}^{\alpha} \right)^{1/2}
\big(
\left( \mathbf{T}^{\alpha} \right)^{1/2}_{\mathcal{S}\mathcal{V}}
\big)^{\dag}
\mathbf{I}_{|\mathcal{S}|}
\Big( \big(
\left( \mathbf{T}^{\alpha} \right)^{1/2}_{\mathcal{S}\mathcal{V}}
\big)^{\mathrm{H}}\Big)^{\dag}
\left( \mathbf{T}^{\alpha} \right)^{1/2},
\]
which simplifies to
\[
\mathbf{E}
=
\left( \mathbf{T}^{\alpha} \right)^{1/2}
\Big( \big(
\left( \mathbf{T}^{\alpha} \right)^{1/2}_{\mathcal{S}\mathcal{V}}\big)^{\mathrm{H}}
\left( \mathbf{T}^{\alpha} \right)^{1/2}_{\mathcal{S}\mathcal{V}}
\Big)^{\dag}
\left( \mathbf{T}^{\alpha} \right)^{1/2}.
\]
Since the filter-based reconstruction operator coincides with the localization-based operator in Eq.~\eqref{xR=TTxS}, both produce the same error covariance $\mathbf{E}$. 

Similar to \cite{GFTSSS,GLCTsampling}, we analyze the error covariance matrix $\mathbf{E}$ in terms of its spectral norm, which can be bounded as
\[
\begin{aligned}
	\| \mathbf{E}\|_{2} 
	&= \Big\| (\mathbf{T}^{\alpha})^{1/2} \Big( (\mathbf{T}^{\alpha})^{-1/2}_{\mathcal{S}\mathcal{V}} (\mathbf{T}^{\alpha})^{1/2}_{\mathcal{S}\mathcal{V}} \Big)^{\dag} (\mathbf{T}^{\alpha})^{1/2} \Big\|_2 \\
	&\le \| (\mathbf{T}^{\alpha})^{1/2} \|_2 \ 
	\Big\| \Big( (\mathbf{T}^{\alpha})^{-1/2}_{\mathcal{S}\mathcal{V}} (\mathbf{T}^{\alpha})^{1/2}_{\mathcal{S}\mathcal{V}} \Big)^{\dag} \Big\|_2 \
	\| (\mathbf{T}^{\alpha})^{1/2} \|_2.
\end{aligned}
\]
Since $\| (\mathbf{T}^{\alpha})^{1/2} \|_2$ is fixed, minimizing the reconstruction error reduces to
\begin{equation}
	\mathcal{S}^{\mathrm{opt}} = \argmin_{\mathcal{S} \subseteq \mathcal{V}} \| \mathbf{E} \|_2 
	= \argmin_{\mathcal{S} \subseteq \mathcal{V}} \Big\| (\mathbf{T}^{\alpha})^{1/2}_{\mathcal{S}\mathcal{V}} \Big)^{\dag} \Big\|_2. \label{norm}
\end{equation}

This objective satisfies
\begin{equation}
	\begin{aligned}
		\Big\| (\mathbf{T}^{\alpha})^{1/2}_{\mathcal{S}\mathcal{V}} \Big)^{\dag} \Big\|_2 
		&\le \Big\| (\mathbf{T}^{\alpha})^{1/2}_{\mathcal{S}\mathcal{V}} \Big)^{\dag} \Big\|_F \\
		&= \sqrt{\mathrm{tr} \Big[ \big( (\mathbf{T}^{\alpha})^{1/2}_{\mathcal{S}\mathcal{V}} \big)^{\dag}  \left( \big( (\mathbf{T}^{\alpha})^{1/2}_{\mathcal{S}\mathcal{V}}\big)^{\dag}  \right)^{\mathrm{H}} \Big]} \\
		&= \sqrt{\mathrm{tr} \Big[ \left( \mathbf{T}^{\alpha}_{\mathcal{S}}\right) ^{-1} \Big]}.
	\end{aligned} \label{trace}
\end{equation}

Similarly, one can express the optimization in terms of the trace, namely
\begin{equation}
	\mathcal{S}^{\mathrm{opt}} = \argmax_{\mathcal{S} \subseteq \mathcal{V}} \Big\| (\mathbf{T}^{\alpha})^{1/2}_{\mathcal{S}\mathcal{V}} \Big\|_2 
	=\argmax_{\mathcal{S} \subseteq \mathcal{V}}  \sqrt{\mathrm{tr} \left[\mathbf{T}^{\alpha}_{\mathcal{S}}\right]}. \label{maxtrace}
\end{equation}

The determinant of the error covariance can be written as 
\[
\begin{aligned}
	\det[\mathbf{E}]=& \det \big[ (\mathbf{T}^{\alpha})^{1/2} (\mathbf{F}^{-\alpha} h^{1/2}(\bm{\Delta}^{\alpha}) \mathbf{F}^{-\alpha}_{\mathcal{S}\mathcal{V}} \mathbf{F}^{\alpha}_{\mathcal{S}\mathcal{V}} h^{1/2}(\bm{\Delta}^{\alpha}) \mathbf{F}^{\alpha} )^{\dag} (\mathbf{T}^{\alpha})^{1/2} \big] \\
	=& \det	\big[ (\mathbf{T}^{\alpha})^{1/2} \mathbf{F}^{-\alpha} \big] 
	\det\big[( h^{1/2}(\bm{\Delta}^{\alpha}_{\mathcal{F}}) \mathbf{F}^{\alpha}_{\mathcal{S}\mathcal{F}} \mathbf{F}^{-\alpha}_{\mathcal{S}\mathcal{F}} h^{1/2}(\bm{\Delta}^{\alpha}_{\mathcal{F}}) )^{\dag}\big] \det\big[ \mathbf{F}^{\alpha} (\mathbf{T}^{\alpha})^{1/2} \big]   \\
	\overset{\text{(a)}}{=}& \det\big[ (\mathbf{T}^{\alpha})^{1/2} \mathbf{F}^{-\alpha}\big]  
	\det\left[ (\mathbf{F}^{-\alpha}_{\mathcal{S}\mathcal{F}} h(\bm{\Delta}^{\alpha}_{\mathcal{F}}) \mathbf{F}^{\alpha}_{\mathcal{S}\mathcal{F}})^{\dag}\right]  
	\det\big[ \mathbf{F}^{\alpha} (\mathbf{T}^{\alpha})^{1/2}\big]  \\
	\overset{\text{(b)}}{=}& \det\big[ (\mathbf{T}^{\alpha})^{1/2} \mathbf{F}^{-\alpha}\big]  
	\det\big[ (\mathbf{T}^{\alpha}_{\mathcal{S}})^{-1}\big]  
	\det\big[ \mathbf{F}^{\alpha} (\mathbf{T}^{\alpha})^{1/2}\big] ,
\end{aligned}
\]
where (a) holds because the two matrices are congruent via $h^{1/2}(\bm{\Delta}^{\alpha}_{\mathcal{F}})$, preserving all nonzero eigenvalues, and (b) follows similarly to the proof of Theorem~\ref{thm3}, assuming negligible components outside $\mathcal{F}$. Consequently, the optimal sampling set with respect to the determinant criterion is
\begin{equation}
	\mathcal{S}^{\mathrm{opt}} = \argmin_{\mathcal{S} \subseteq \mathcal{V}} \det[\mathbf{E}] = \argmin_{\mathcal{S} \subseteq \mathcal{V}} \det\left[ (\mathbf{T}^{\alpha}_{\mathcal{S}})^{-1}\right] . \label{det}
\end{equation}

Therefore, minimizing the reconstruction error can be achieved by selecting sampling strategies that optimize matrix-based criteria such as the spectral norm \eqref{norm}, trace \eqref{trace}, or determinant \eqref{det}.

\section{Optimal Fractional Sampling Strategies}
\label{OptimalSampling}

Effective sampling of graph signals requires not only determining the number of samples but also selecting informative sampling nodes, as their locations directly affect reconstruction quality. In the GFRFT domain, sampling similarly depends on the choice of an optimal sampling set. Extending the reconstruction error minimization strategy in~\cite{GFRFTsampling}, this section considers two additional criteria, namely maximizing the graph fractional cutoff frequency and maximizing localization based information content. These strategies are unified within a localization framework.

\subsection{GFRFT Sampling Based on Cutoff Frequency}
In classical graph sampling theory, the bandwidth is determined by low-order eigenvalues of the graph Laplacian matrix. In the GFRFT framework, by selecting the graph Laplacian as the graph shift operator, the graph fractional Laplacian matrix \cite{GFRFT_unified} can be derived from Eq. \eqref{Salpha} as
\begin{equation}
	\mathbf{L}^{\alpha} := \mathbf{F}^{-\alpha}\bm{\Delta}^{\alpha}\mathbf{F}^{\alpha}.
\end{equation}
This operator incorporates both fractional eigenvalues and GFRFT-based eigenvectors\footnote{Unlike the classical construction $\mathbf{U}\bm{\Delta}^{\alpha}\mathbf{U}^{\mathrm{H}}$, the basis is given by the GFRFT matrix $\mathbf{F}^{\alpha}$.}, producing a spectrum whose structure varies with the order~$\alpha$ and better reflects the intrinsic smoothness properties of graph signals.

To evaluate a sampling set $\mathcal{S}$, we first determine its associated cutoff frequency, defined as the largest $\omega$ for which $\mathcal{S}$ guarantees uniqueness in the graph fractional bandlimited space. Following~\cite{GFTefficient}, $\omega(\phi)$ is approximated by $\omega_{k}(\phi)$ and we define the $k$-th order cutoff estimate as
\[
\Omega_{k}(\mathcal{S})
:= \min_{\phi \in \mathcal{D}^c} \omega_k (\phi)
=  \min_{\phi \in \mathcal{D}^c}
\left(
\frac{\|(\mathbf{L}^{\alpha})^{k} \phi\|}
{\|\phi\|}
\right)^{1/k},
\]
where $\mathcal{D}^c$ denotes the complement of $\mathcal{D}$, and the classical operator $\mathbf{L}^{k}$ is replaced by its GFRFT-based counterpart to characterize smoothness in the fractional domain. This yields
\[
\left((\mathbf{L}^{\alpha})^{\mathrm{H}}\right)^{k}
(\mathbf{L}^{\alpha})^{k}
= (\mathbf{L}^{\alpha})^{2k}.
\]

Therefore, the cutoff frequency is equivalently obtained by minimizing the Rayleigh quotient over its complement~$\mathcal{S}^c$,
\[
\Omega_{k}(\mathcal{S})
=
\Bigg[
\min_{\psi \neq 0}
\frac{
	\psi^{\mathrm{H}}
	(\mathbf{L}^{\alpha})^{2k}_{\mathcal{S}^c}\,
	\psi
}{
	\psi^{\mathrm{H}}\psi
}
\Bigg]^{1/2k}
=
\Big(
\lambda_{\min}\big((\mathbf{L}^{\alpha})^{2k}_{\mathcal{S}^c}\big)
\Big)^{1/2k},
\]
where the associated eigenvector $\psi_{\min,k}$ corresponds to the smoothest graph fractional signal that cannot be uniquely reconstructed under the sampling set~$\mathcal{S}$. We define
\[
\phi^{\ast}_{k}(i)=
\begin{cases}
	\psi_{\min,k}(i), & i\in \mathcal{S}^c,\\
	0, & i\in \mathcal{S}.
\end{cases}
\]
The constructed signal characterizes the worst-case ambiguity associated with the current sampling set, and its largest-energy vertex $|\phi^{\ast}_{k}(i)|^{2}$ identifies the most effective location to improve recoverability by increasing $\Omega_{k}(\mathcal{S})$ and suppressing indistinguishable modes. This naturally leads to a greedy strategy in which $(\mathbf{L}^{\alpha})^{2k}$ is formed at each iteration, its restriction to $\mathcal{S}^c$ is extracted, the smallest eigenvector is obtained, and the vertex with maximal energy is included in $\mathcal{S}$. Since fractional powers are applied to both eigenvalues and eigenvectors in $\mathbf{L}^{\alpha}$, the resulting sampling set better captures the intrinsic fractional smoothness of graph fractional bandlimited signals. 

\textit{Maximizing cutoff frequency (MaxCut):} The sampling set is chosen to minimize the energy of the smoothest unrecoverable fractional mode on the unsampled vertices, leading to the optimization problem as follows
\begin{equation}
	\mathcal{S}^{\mathrm{opt}} = \argmax_{\mathcal{S} \subseteq \mathcal{V}}~
	\lambda_{\min}\Big( (\mathbf{L}^{\alpha})^{2k}_{\mathcal{S}^c} \Big), \label{S0}
\end{equation}
with the next vertex at iteration $m$ selected as
\begin{equation}
	y^{\mathrm{opt}} =\argmax_{y \in \mathcal{S}^c_m}~ \lambda_{\min}\Big( (\mathbf{L}^{\alpha})^{2k}_{\mathcal{S}_m^c} \Big)(y). 
	\label{y0}
\end{equation}
Since the graph fractional localization operator $\mathbf{T}^{\alpha}$ with $h(\bm{\Delta}^{\alpha})=\bm{\Delta}^{\alpha}$ is defined analogously to $\mathbf{L}^{\alpha}$ in Eq.~\eqref{Talpha}, the same objective can be expressed as
\[
\mathcal{S}^{\mathrm{opt}} 
= \argmax_{\mathcal{S} \subseteq \mathcal{V}}~
\lambda_{\min}\Big( (\mathbf{T}^{\alpha})^{2k}_{\mathcal{S}^c} \Big) = \argmin_{\mathcal{S} \subseteq \mathcal{V}}
\Big\| \big((\mathbf{T}^{\alpha})^{2k}_{\mathcal{S}^c}\big)^{-1} \Big\|_2.
\]

\subsection{GFRFT Sampling Based on Error Minimization}

When the bandwidth satisfies $|\mathcal{F}| \leq |\mathcal{S}|$, the sampling set can be chosen by minimizing the reconstruction error \cite{GFTsampling}. Although the optimal sampling strategy can be directly derived using the localization operators in Eqs.~\eqref{e=TTxi}--\eqref{det}, we employ the GFRFT operator for theoretical completeness. For the noisy observation modeled as
\[
\bm{x}_{\mathcal{R}} = \mathbf{R} (\mathbf{D}\bm{x} + \bm{\xi}),
\]
the resulting error is, according to Eq.~\eqref{R},
\[
\bm{e} = \mathbf{R}\bm{\xi}
= \mathbf{F}^{-\alpha}_{\mathcal{V}\mathcal{F}}
\left(\mathbf{F}^{-\alpha}_{\mathcal{S}\mathcal{F}}\right)^{\dagger} \bm{\xi}. 
\]

\textit{Maximizing singular value of the minimum (MaxSigMin):} Minimizing the $\ell_2$-norm of the error yields the optimization problem \cite{GFRFTsampling},
\[
\mathcal{S}^{\mathrm{opt}}
= \argmin_{\mathcal{S} \subseteq \mathcal{V}}\|\bm{e}\|_2
= \argmin_{\mathcal{S} \subseteq \mathcal{V}}\left\|
\mathbf{F}^{-\alpha}_{\mathcal{V}\mathcal{F}}
\left(\mathbf{F}^{-\alpha}_{\mathcal{S}\mathcal{F}}\right)^{\dagger}
\bm{\xi}
\right\|_2.
\]

Using the Cauchy--Schwarz inequality, the upper bound of the error norm becomes
\[
\|\bm{e}\|_2
\le
\left\| \mathbf{F}^{-\alpha}_{\mathcal{V}\mathcal{F}} \right\|_2
\cdot
\left\|
\left(\mathbf{F}^{-\alpha}_{\mathcal{S}\mathcal{F}}\right)^{\dagger}
\right\|_2
\cdot
\|\bm{\xi}\|_2.
\]
Since $\| \mathbf{F}^{-\alpha}_{\mathcal{V}\mathcal{F}} \|_2$ and $\|\bm{\xi}\|_2$ are constants with respect to $\mathcal{S}$, the objective reduces to minimizing the spectral norm of the pseudo-inverse. This is equivalent to maximizing the smallest singular value of $\mathbf{F}^{-\alpha}_{\mathcal{S}\mathcal{F}}$, namely
\begin{equation}
	\mathcal{S}^{\mathrm{opt}}
	=
	\argmax_{\mathcal{S} \subseteq \mathcal{V}}~
	\sigma_{\min}\!\left( \mathbf{F}^{-\alpha}_{\mathcal{S}\mathcal{F}} \right).
	\label{S1}
\end{equation}
This optimization problem can be solved efficiently via a greedy strategy. At the $m$-th iteration, the next vertex $y^{\mathrm{opt}}$ is selected as
\begin{equation}
	y^{\mathrm{opt}}
	= \argmin_{y \in \mathcal{S}_m^{c}}~
	\sigma_{\min}\left( \mathbf{F}^{-\alpha}_{(\mathcal{S}_m\cup y) \mathcal{F}} \right).
	\label{y1}
\end{equation}

Based on Eq.~\eqref{S1}, the optimal sampling set can be reformulated using the localization operator
\[
\begin{aligned}
	\mathcal{S}^{\mathrm{opt}}
	=& \argmax_{\mathcal{S} \subseteq \mathcal{V}}~
	\sigma_{\min}\!\left( \mathbf{F}^{-\alpha}_{\mathcal{S}\mathcal{F}} \right) \\
	=& \argmax_{\mathcal{S} \subseteq \mathcal{V}}~
	\sigma_{\min}\!\left( \mathbf{F}^{\alpha}_{\mathcal{F}\mathcal{S}} \right) \\
	=& \argmax_{\mathcal{S} \subseteq \mathcal{V}}~
	\sigma_{\min}\!\left( \mathbf{F}^{-\alpha}
	\boldsymbol{\Sigma}
	\mathbf{F}^{\alpha}
	\mathbf{D}
	\right) \\
	=& \argmin_{\mathcal{S} \subseteq \mathcal{V}}
	\left\|
	\left( \mathbf{T}^{\alpha}_{\mathcal{V}\mathcal{S}} \right)^{\dagger}
	\right\|_2,
\end{aligned}
\]
where $\mathbf{T}^{\alpha}$ denotes the spectral localization operator with $h(\boldsymbol{\Delta}^{\alpha}) = \boldsymbol{\Sigma}$. Therefore, MaxSigMin is equivalent to the classical E-optimal design criterion \cite{Design}.

\textit{Minimizing trace (MinTrac):} 
The objective is to minimize the trace of the error covariance matrix, which corresponds to minimizing the total reconstruction error energy. Using the error covariance derived from Eq.~\eqref{R}, the optimization problem is formulated as
\begin{equation}
	\begin{aligned}
		\mathcal{S}^{\mathrm{opt}}
		& = \argmin_{\mathcal{S} \subseteq \mathcal{V}}~ \mathrm{tr}(\bm{e}\bm{e}^{\mathrm{H}}) \\
		& = \argmin_{\mathcal{S} \subseteq \mathcal{V}}~
		\mathrm{tr}\!\left[
		\left(\mathbf{F}^{\alpha}_{\mathcal{S}\mathcal{F}}\right)^{\dagger}
		\mathbf{F}^{\alpha}_{\mathcal{V}\mathcal{F}}
		\mathbf{F}^{-\alpha}_{\mathcal{V}\mathcal{F}}
		\left(\mathbf{F}^{-\alpha}_{\mathcal{S}\mathcal{F}}\right)^{\dagger}
		\right] \\
		& = \argmin_{\mathcal{S} \subseteq \mathcal{V}}~
		\mathrm{tr}\!\left[
		\left(
		\mathbf{F}^{\alpha}_{\mathcal{S}\mathcal{F}}
		\mathbf{F}^{-\alpha}_{\mathcal{S}\mathcal{F}}
		\right)^{-1}
		\right].
	\end{aligned}
	\label{SJ2}
\end{equation}
Thus, the optimal sampling set is obtained by minimizing the trace of the inverse restricted spectral operator via a greedy procedure. At the $m$-th iteration, the vertex selected is
\begin{equation}
	y^{\mathrm{opt}}
	=
	\argmin_{y \in \mathcal{S}^{c}_{m}}~
	\mathrm{tr}\!\left[
	\left(
	\mathbf{F}^{\alpha}_{(\mathcal{S}_{m} \cup y)\mathcal{F}}
	\mathbf{F}^{-\alpha}_{(\mathcal{S}_{m} \cup y)\mathcal{F}}
	\right)^{-1}
	\right].
	\label{y2}
\end{equation}

The trace-minimization problem in \eqref{SJ2} can be equivalently written as
\[
\mathcal{S}^{\mathrm{opt}}
= \argmin_{\mathcal{S} \subseteq \mathcal{V}}~
\mathrm{tr}\!\left[
\left(
\mathbf{F}^{\alpha}_{\mathcal{S}\mathcal{F}}
\mathbf{F}^{-\alpha}_{\mathcal{S}\mathcal{F}}
\right)^{-1}
\right]
= \argmin_{\mathcal{S} \subseteq \mathcal{V}}~
\mathrm{tr}\!\left[
\left(
\mathbf{T}^{\alpha}_{\mathcal{S}}
\right)^{-1}
\right],
\]
where the operator $\mathbf{T}^{\alpha}$ corresponds to the choice $h(\boldsymbol{\Delta}^{\alpha})=\boldsymbol{\Sigma}$, similar to the MaxSigMin. Under this formulation, MinTrac is equivalent to the A-optimal design criterion.

\subsection{GFRFT Sampling Based on Localized Basis}
\label{Sec:4.3}
Similarly, when $|\mathcal{F}| \leq |\mathcal{S}|$, the signal $\boldsymbol{x}$ can be reconstructed by using the vertex-limiting operator $\mathbf{D}$ together with the graph fractional bandlimiting operator $\mathbf{B}^{\alpha}$. Perfect recovery is guaranteed when $\boldsymbol{x}$ satisfies Theorem~\ref{thm1}. Under the localized basis framework, the reconstruction error is given by
\[
\boldsymbol{e}
= \mathbf{R}\boldsymbol{\xi}
= \left( \mathbf{D} \mathbf{B}^{\alpha} \right)^{\dagger} \boldsymbol{\xi}
= \left( \mathbf{B}^{\alpha} \mathbf{D} \mathbf{B}^{\alpha} \right)^{\dagger} \mathbf{D}\boldsymbol{\xi}.
\]

\textit{Minimizing the Frobenius norm of the pseudo-inverse (MinPinv):}
Analogous to Eq.~\eqref{trace}, this criterion selects the sampling set that minimizes the Frobenius norm of the reconstruction operator,
\[
\mathcal{S}^{\mathrm{opt}}
= \argmin_{\mathcal{S} \subseteq \mathcal{V}}
\left\|
\left( \mathbf{B}^{\alpha} \mathbf{D} \mathbf{B}^{\alpha} \right)^{\dagger}
\mathbf{D}\boldsymbol{\xi}
\right\|_{F}.
\]
By applying the Cauchy--Schwarz inequality, the above quantity admits the upper bound
\[
\left\|
\left( \mathbf{B}^{\alpha} \mathbf{D} \mathbf{B}^{\alpha} \right)^{\dagger}
\mathbf{D}\boldsymbol{\xi}
\right\|_{F}
\leq
\left\|
\left( \boldsymbol{\Sigma}\mathbf{F}^{\alpha}\mathbf{D} \right)^{\dagger}
\right\|_{F}
\cdot
\left\|\mathbf{D}\boldsymbol{\xi}\right\|_{F}.
\]
To make the dependence on the sampling set explicit, we restore the corresponding indices. Since $\left\|\mathbf{D}\boldsymbol{\xi}\right\|_{F}$ is constant, the optimization reduces to
\begin{equation}
	\begin{aligned}
		\mathcal{S}^{\mathrm{opt}}
		=&~
		\argmin_{\mathcal{S} \subseteq \mathcal{V}}
		\left\|
		\left(
		\boldsymbol{\Sigma}_{\mathcal{F}}
		\mathbf{F}^{\alpha}
		\mathbf{D}_{\mathcal{S}}
		\right)^{\dagger}
		\right\|_{F}
		\\[2mm]
		=&~
		\argmin_{\mathcal{S} \subseteq \mathcal{V}}
		\sum_{i=1}^{|\mathcal{F}|}
		\frac{1}{\sigma_i\!\left(
			\boldsymbol{\Sigma}_{\mathcal{F}}
			\mathbf{F}^{\alpha}
			\mathbf{D}_{\mathcal{S}}
			\right)} .
	\end{aligned}
	\label{SJ3}
\end{equation}
Following this, the vertex-selection rule becomes
\begin{equation}
	y^{\mathrm{opt}}
	= \argmin_{y \in \mathcal{S}_m^{c}}
	\sum_{i=1}^{|\mathcal{F}|}
	\frac{1}{\sigma_i\!\left(
		\boldsymbol{\Sigma}_{\mathcal{F}}
		\mathbf{F}^{\alpha}
		\mathbf{D}_{\mathcal{S}_m \cup y}
		\right)} .
	\label{y3}
\end{equation}

According to the above Eq. \eqref{SJ3}, the criterion based on pseudo-inverse can be further localized as
\[
\begin{aligned}
	\mathcal{S}^{\mathrm{opt}}
	=&~
	\argmin_{\mathcal{S} \subseteq \mathcal{V}}
	\sum_{i=1}^{|\mathcal{F}|}
	\frac{1}{\sigma_i\!\left(
		\boldsymbol{\Sigma}_{\mathcal{F}}
		\mathbf{F}^{\alpha}
		\mathbf{D}_{\mathcal{S}}
		\right)} 
	\\[1mm]
	=&~
	\argmin_{\mathcal{S} \subseteq \mathcal{V}}
	\left\|
	\left( \mathbf{T}^{\alpha}_{\mathcal{V}\mathcal{S}} \right)^{\dagger}
	\right\|_{F}
	= 
	\argmin_{\mathcal{S} \subseteq \mathcal{V}}~
	\mathrm{tr}\!\left[
	\left( \mathbf{T}^{\alpha}_{\mathcal{S}} \right)^{-1}
	\right],
\end{aligned}
\label{S3}
\]
which is consistent with the error minimization principle, where
$\mathbf{T}^{\alpha} = \mathbf{F}^{-\alpha}\boldsymbol{\Sigma}_{\mathcal{F}} \mathbf{F}^{\alpha}$. It is worth noting that, after localization, the MinPinv criterion leads to the same selection rule for MinTrac, indicating that MinPinv is also equivalent to the A-optimal design.

\textit{Maximizing singular values (MaxSig):}
In contrast to the MinPinv strategy, this criterion seeks to maximize the Frobenius norm of 
$\mathbf{B}^{\alpha} \mathbf{D} \mathbf{B}^{\alpha}$, thereby enhancing numerical stability,
\begin{equation}
	\begin{aligned}
		\mathcal{S}^{\mathrm{opt}}
		=&~ \argmax_{\mathcal{S} \subseteq \mathcal{V}}
		\left\|
		\boldsymbol{\Sigma}_{\mathcal{F}}
		\mathbf{F}^{\alpha}
		\mathbf{D}_{\mathcal{S}}
		\right\|_{F}
		\\
		=&~
		\argmax_{\mathcal{S} \subseteq \mathcal{V}}
		\sum_{i=1}^{|\mathcal{F}|}
		\sigma_i\!\left(
		\boldsymbol{\Sigma}_{\mathcal{F}}
		\mathbf{F}^{\alpha}
		\mathbf{D}_{\mathcal{S}}
		\right).
	\end{aligned}
	\label{SJ4}
\end{equation}
and thus a greedy procedure is again employed, selecting at each iteration the vertex that maximizes
\begin{equation}
	y^{\mathrm{opt}}
	= \argmax_{y \in \mathcal{S}^{c}_{m}}
	\sum_{i=1}^{|\mathcal{F}|}
	\sigma_i\!\left(
	\boldsymbol{\Sigma}_{\mathcal{F}}
	\mathbf{F}^{\alpha}
	\mathbf{D}_{\mathcal{S}_{m} \cup y}
	\right).
	\label{y4}
\end{equation}

Its localized formulation becomes
\[
\mathcal{S}^{\mathrm{opt}}
= \argmax_{\mathcal{S} \subseteq \mathcal{V}}
\left\|
\mathbf{T}^{\alpha}_{\mathcal{V}\mathcal{S}}
\right\|_{F}
= \argmax_{\mathcal{S} \subseteq \mathcal{V}}
\mathrm{tr}\!\left[
\mathbf{T}^{\alpha}_{\mathcal{S}}
\right],
\label{S4}
\]
hence, MaxSig criterion is equivalent to Eq.~\eqref{maxtrace}, and therefore corresponds to the T-optimal design.

\textit{Maximizing the volume of the parallelepiped (MaxVol):}
This criterion focuses on maximizing the volume spanned by the selected rows of the transform matrix. Analogous to Eq.~\eqref{det}, this can be achieved by maximizing the determinant,
\begin{equation}
	\begin{aligned}
		\mathcal{S}^{\mathrm{opt}}
		= & \argmax_{\mathcal{S} \subseteq \mathcal{V}}
		~\det\!\left[
		\mathbf{F}^{-\alpha}_{\mathcal{S}\mathcal{F}}
		\mathbf{F}^{\alpha}_{\mathcal{S}\mathcal{F}}
		\right]\\
		= &\argmax_{\mathcal{S} \subseteq \mathcal{V}}
		~\det\!\left[
		\left(
		\mathbf{F}^{-\alpha}
		\boldsymbol{\Sigma}_{\mathcal{F}}
		\mathbf{F}^{\alpha}
		\right)_{\mathcal{S}}
		\right],
	\end{aligned}\label{SJ5}
\end{equation}
and a greedy selection procedure is again applied, where at the $m$-th iteration the chosen vertex satisfies
\begin{equation}
	y^{\mathrm{opt}}
	= \argmax_{y \in \mathcal{S}^{c}_{m}}~\det\!\left[
	\left(
	\mathbf{F}^{-\alpha}
	\boldsymbol{\Sigma}_{\mathcal{F}}
	\mathbf{F}^{\alpha}
	\right)_{\mathcal{S}_{m} \cup y}
	\right].
	\label{y5}
\end{equation}

The volume maximization in \eqref{SJ5} can be further rewritten in terms of the localized operator as
\[
\mathcal{S}^{\mathrm{opt}}
= \argmax_{\mathcal{S} \subseteq \mathcal{V}}
~\det\!\left[
\mathbf{T}^{\alpha}_{\mathcal{S}}
\right],
\label{S5}
\]
thus MaxVol corresponds precisely to the D-optimal design.

Eqs. \eqref{S0}–\eqref{y5} unify the six GFRFT-based sampling strategies within the localized filtering framework defined by the operator $\mathbf{T}^{\alpha}$, where each sampling rule is determined by the choice of spectral kernel $ h(\boldsymbol{\Delta}^{\alpha})$. The selected kernel corresponds to a specific spectral bandlimiting operator, thereby linking each sampling operator to its localized formulation. A summary of the associated objective functions is provided in Table~\ref{SMOFLO}.
\begin{table*}[t]
	\caption{Sampling Method with Objective Function and Localization Operator}
	\label{SMOFLO}
	\scriptsize
	\centering
	\begin{tabular}{ccll}
		\toprule
		Method & Optimal Design&  \ \ \ \ Objective Function& Localized Filter Operator \\
		\midrule
		MaxCut& E-optimal &$\argmax_{\mathcal{S} \subseteq \mathcal{V}}~
		\lambda_{\min}\Big( (\mathbf{L}^{\alpha})^{2k}_{\mathcal{S}^c} \Big)$  &$\argmin_{\mathcal{S} \subseteq \mathcal{V}}~\Big\| \big((\mathbf{T}^{\alpha})^{2k}_{\mathcal{S}^c}\big)^{-1} \Big\|_2$ \vspace{0.15cm}\\
		MaxSigMin& E-optimal &$\argmax_{\mathcal{S} \subseteq \mathcal{V}}~\sigma_{\min} \big( \mathbf{F}^{-\alpha}_{\mathcal{S} \mathcal{F}} \big)$  &$\argmin_{\mathcal{S} \subseteq \mathcal{V}}~ \Big\| \big( \mathbf{T}^{\alpha}_{\mathcal{V} \mathcal{S}} \big)^{\dagger} \Big\|_2$\vspace{0.15cm}\\
		MinTrac & A-optimal &$\argmin_{\mathcal{S} \subseteq \mathcal{V}}~\mathrm{tr} \Big[ \big( \mathbf{F}^{\alpha}_{\mathcal{S}\mathcal{F}} \mathbf{F}^{-\alpha}_{\mathcal{S}\mathcal{F}} \big)^{-1} \Big]$&$\argmin_{\mathcal{S} \subseteq \mathcal{V}}~ \mathrm{tr} \Big[ \big( \mathbf{T}^{\alpha}_{\mathcal{S}} \big)^{-1} \Big]$\vspace{0.15cm}\\
		MinPinv& A-optimal &$\argmin_{\mathcal{S} \subseteq \mathcal{V}}~\sum_{i=1}^{|\mathcal{F}|} 1/\sigma_i\left(\boldsymbol{\Sigma}_{\mathcal{F}} \mathbf{F}^{\alpha} \mathbf{D}_{\mathcal{S}}\right)$&$\argmin_{\mathcal{S} \subseteq \mathcal{V}}~ \mathrm{tr} \Big[ \big( \mathbf{T}^{\alpha}_{\mathcal{S}} \big)^{-1} \Big]$\vspace{0.15cm}\\
		MaxSig& T-optimal &$\argmax_{\mathcal{S} \subseteq \mathcal{V}}~\sum_{i=1}^{|\mathcal{F}|} \sigma_i\left(\boldsymbol{\Sigma}_{\mathcal{F}} \mathbf{F}^{\alpha} \mathbf{D}_{\mathcal{S}}\right)$ & $\argmax_{\mathcal{S} \subseteq \mathcal{V}}~ \mathrm{tr} \left[ \mathbf{T}^{\alpha}_{\mathcal{S}} \right]$\vspace{0.15cm}\\
		MaxVol& D-optimal &$ \argmax_{\mathcal{S} \subseteq \mathcal{V}}~\det \left[\left(  \mathbf{F}^{-\alpha} \boldsymbol{\Sigma}_{\mathcal{F}} \mathbf{F}^{\alpha}\right) _{\mathcal{S}} \right] $ &$ \argmax_{\mathcal{S} \subseteq \mathcal{V}}~ \det \left[ \mathbf{T}^{\alpha}_{\mathcal{S}} \right]$\\
		\bottomrule
	\end{tabular}
\end{table*}

\section{Fast Sampling for Graph Fractional Domain}
\label{FastSampling}
Although the optimization criteria summarized in Table~\ref{SMOFLO} effectively address the selection of optimal sampling sets in the GFRFT domain, maximizing or minimizing these cost functions requires repeated eigenvalue decompositions, together with iterative greedy updates, resulting in relatively high computational complexity. Inspired by \cite{GFTSSS}, we exploit the localized filtering operator defined in Eq.~\eqref{Talpha} to reduce the computational burden and accelerate the sampling procedure.

\subsection{GFRFT Sampling Based on Graph Fractional Localization Operator}
Missing samples are interpolated by the localized operator $\mathbf{T}^{\alpha}$ evaluated at the observed vertices. For a given vector $\bm{T}^{\alpha}_{h,i}$, its \textit{support} set indicates the region over which vertex $i$ can interpolate unobserved samples. The objective is therefore to maximize the ``coverage area'' of the sampling set $\mathcal{S}$, quantified by $\mathcal{C} \big(|\bm{T}^{\alpha}_{h,i}|\big)$, which leads to maximizing $\sum_{i\in \mathcal{S}} \mathcal{C} \big(|\bm{T}^{\alpha}_{h,i}|\big)$.

In other words, we aim to select the sampling set that achieves the \textit{maximizing coverage area (MaxCov)}, such that each operator $\bm{T}^{\alpha}_{h,i}$ exhibits sufficiently large energy in its primary support, while the overlap between the supports of different vertices $i \neq j$ remains small. Consequently, the corresponding optimization problem can be written as
\begin{equation}
	\mathcal{S}^{\mathrm{opt}}
	= \argmax_{\mathcal{S}\subseteq \mathcal{V}}
	\sum_{i\in \mathcal{S}}
	\Big\langle
	\Big( \epsilon \mathbf{1}_{N} - \sum_{j\in \mathcal{S},\, j\neq i} \big| \bm{T}^{\alpha}_{h,j} \big| \Big),
	\ \big| \bm{T}^{\alpha}_{h,i} \big|
	\Big\rangle ,
\end{equation}
where $\mathbf{1}_{N}$ denotes the all-ones vector of length $N$, and the scalar parameter $\epsilon = \frac{1}{N} \sum_{i \in \mathcal{V}} \sum_{j \in \mathcal{S}} \big| T^{\alpha}_{h,j}(i) \big|$ is empirically chosen. Following \cite{GFTSSS}, we employ a greedy algorithm to optimize this cost. At iteration $m$, the algorithm selects the vertex
\begin{equation}
	y^{\mathrm{opt}}
	= \argmax_{y\in \mathcal{S}^c_m}
	\Big\langle
	\Big( \epsilon \mathbf{1}_{N} - \sum_{j\in \mathcal{S}_{m}} \big| \bm{T}^{\alpha}_{h,j} \big| \Big)_{+},
	\ \big| \bm{T}^{\alpha}_{h,y} \big|
	\Big\rangle , \label{yopt}
\end{equation}
where $(\cdot)_{+}$ denotes the elementwise positive-part operator, which sets negative entries to zero while preserving nonnegative values.

In computing the weighted norm of $\bm{T}^{\alpha}_{h,y}$ in \eqref{yopt}, the contribution of $T^{\alpha}_{h,y}(i)$ is assigned a smaller weight if vertex $i$ is already covered, whereas entries corresponding to previously selected vertices $ j \in \mathcal{S}_m$ receive larger weights. This behavior arises because the kernel $h$ is a polynomial function.

\subsection{Computational Complexity and Algorithm}
Table~\ref{Computation} summarizes the computational complexity of seven GFRFT sampling strategies, including six representative methods and the proposed fast MaxCov approach. Here, $N$ denotes the size of the vertex domain, $|\mathcal{F}|$ the signal bandwidth in the GFRFT domain, $|\mathcal{S}|$ the number of selected vertices, and $P$ the number of nonzero entries in the localized operator $\mathbf{T}^{\alpha}$. Operator computations involve the construction of graph fractional operators through matrix eigendecomposition, which for all methods scales as $\mathcal{O}(N^3)$ due to the full spectral decomposition required by the GFRFT.
\begin{table}[h!]
	\centering
	\caption{Computational Complexity of GFRFT Sampling Strategies}
	\label{Computation}
	\footnotesize
	\begin{tabular}{lcc}
		\toprule
		Method & Operator Computations & Sampling Set Selection \\
		\midrule
		MaxCut     & \multirow{7}{*}{$\mathcal{O}(N^3)$}  & $\mathcal{O}(N^2 |\mathcal{S}|)$  \vspace{0.15cm}\\
		MaxSigMin  & & $\mathcal{O}(N |\mathcal{S}|^2 |\mathcal{F}|^2)$  \vspace{0.15cm}\\
		MinTrac    & & $\mathcal{O}(N |\mathcal{S}|^4)$  \vspace{0.15cm}\\
		MinPinv    & & $\mathcal{O}(N |\mathcal{S}|^2 |\mathcal{F}|^2)$ \vspace{0.15cm} \\
		MaxSig     & & $\mathcal{O}(N |\mathcal{S}|^2 |\mathcal{F}|^2)$  \vspace{0.15cm}\\
		MaxVol     & & $\mathcal{O}(N |\mathcal{S}||\mathcal{F}|^3)$  \vspace{0.15cm}\\
		MaxCov     & &$\mathcal{O}(P |\mathcal{S}|)$ \\
		\bottomrule
	\end{tabular}
\end{table}

However, the subsequent sampling set selection varies significantly across methods. MaxCut, MaxVol, and MinTrac require repeated eigenvalue or matrix inversion operations on dense matrices whose size grows with the sampling set, while MaxSigMin, MinPinv, and MaxSig iteratively evaluate singular values or pseudo-inverses of submatrices. In contrast, MaxCov exploits the sparsity and localization of the Hermitian operator $\mathbf{T}^{\alpha}$, allowing greedy selection using simple vector operations over only $P$ nonzero entries per iteration. This explains the observed efficiency of MaxCov, which achieves fast runtime while maintaining competitive reconstruction performance compared to other GFRFT sampling strategies.

Precise localization in both the vertex and graph fractional spectral domains allows signal energy to be concentrated on specific nodes. Maximizing this energy over the sampling set $\mathcal{S}$ ensures that key local structures and dynamic patterns are effectively captured. Across the entire vertex set $\mathcal{V}$, a carefully selected $\mathcal{S}$ preserves the essential characteristics of the graph signal. This sampling strategy exhibits submodular behavior, with diminishing marginal gains as the set grows, making the use of a greedy algorithm both natural and effective \cite{Greedy}.

Several criteria, such as MinTrac, MinPinv, and MaxSig, involve the maximization or minimization of the trace of certain matrices. Due to their shared structural properties, these criteria can be addressed within a unified algorithmic framework, as summarized in Algorithm~\ref{alg01}.

\begin{algorithm}[h!]
	\footnotesize
	\caption{Unified Greedy Sampling Algorithm for GFRFT-Based Methods}
	\label{alg01}
	\begin{algorithmic}
		\STATE \textbf{Input:} Matrices $\mathbf{F}^{-\alpha}_{\mathcal{V}\mathcal{F}}$, $\mathbf{L}^{\alpha}$, $\mathbf{T}^{\alpha}$, number of samples $|\mathcal{S}|$
		\STATE \textbf{Output:} Sampling set $\mathcal{S}$
		\STATE Initialize $\mathcal{S}\leftarrow\emptyset$
		\FOR{$m=1:|\mathcal{S}|$}
		\FOR{$y=1:N$ and $y\notin\mathcal{S}$}
		\STATE Evaluate the objective according to the chosen criterion:
		\begin{itemize}
			\item[] $\bullet$ MaxCut: Eq.~\eqref{y0}
			\item[] $\bullet$ MaxSigMin: Eq.~\eqref{y1}
			\item[] $\bullet$ MinTrac: Eq.~\eqref{y2}
			\item[] $\bullet$ MinPinv: Eq.~\eqref{y3}
			\item[] $\bullet$ MaxSig: Eq.~\eqref{y4}
			\item[] $\bullet$ MaxVol: Eq.~\eqref{y5}
			\item[] $\bullet$ MaxCov: Eq.~\eqref{yopt}
		\end{itemize}
		\ENDFOR
		\STATE Select $y^{\mathrm{opt}}$ according to the optimality rule of the chosen equation
		\STATE Update $\mathcal{S}\leftarrow \mathcal{S}\cup y^{\mathrm{opt}}$
		\ENDFOR
		\RETURN $\mathcal{S}$
	\end{algorithmic}
\end{algorithm}

\section{Experiments}
\label{Experiments}
This section presents extensive numerical experiments on both synthetic and real-world datasets to evaluate the performance of the proposed GFRFT-based sampling strategies. Specifically, we first conduct comparative studies on randomly generated graphs and the real graph, benchmarking six representative sampling set selection methods against random sampling. Furthermore, on selected graphs, we investigate the performance of fast sampling schemes to assess the trade-off between computational efficiency and reconstruction accuracy. Finally, to examine the practical impact of GFRFT-based sampling, we apply the proposed methods to active semi-supervised learning of graph signals and to EEG signal representation tasks. Unless otherwise stated, the GFRFT is constructed using the Laplacian matrix, and all experiments are implemented within the GSP Toolbox framework \cite{GSPBOX}.

\subsection{Comparison of GFRFT Sampling Methods}
To evaluate the effectiveness of the proposed sampling strategies under the GFRFT framework, numerical experiments are conducted on two representative random graph models, namely a sensor graph and an Erd\H{o}s-R\'enyi graph. Both graphs consist of $N=200$ nodes, and the Erd\H{o}s-R\'enyi graph is generated with an edge probability of $0.05$. For each graph, the GFRFT is constructed with a fixed fractional order $\alpha = 0.7$. The bandwidth is set to $|\mathcal{F}| = 40$, based on which the GFRFT bandlimited operator $\mathbf{B}^{\alpha}$ is defined. The graph signal is then generated as $\bm{f} = \mathbf{B}^{\alpha} \,[\bm{1}_{N/2},-\bm{1}_{N/2}]^{\top}$, yielding a noise-free graph signal that is graph fractional bandlimited. Examples of the constructed signals on the two graph models are illustrated in Fig.~\ref{Randomsignal}.
\begin{figure}[h!]
	\centering
	\includegraphics[width=0.9\linewidth]{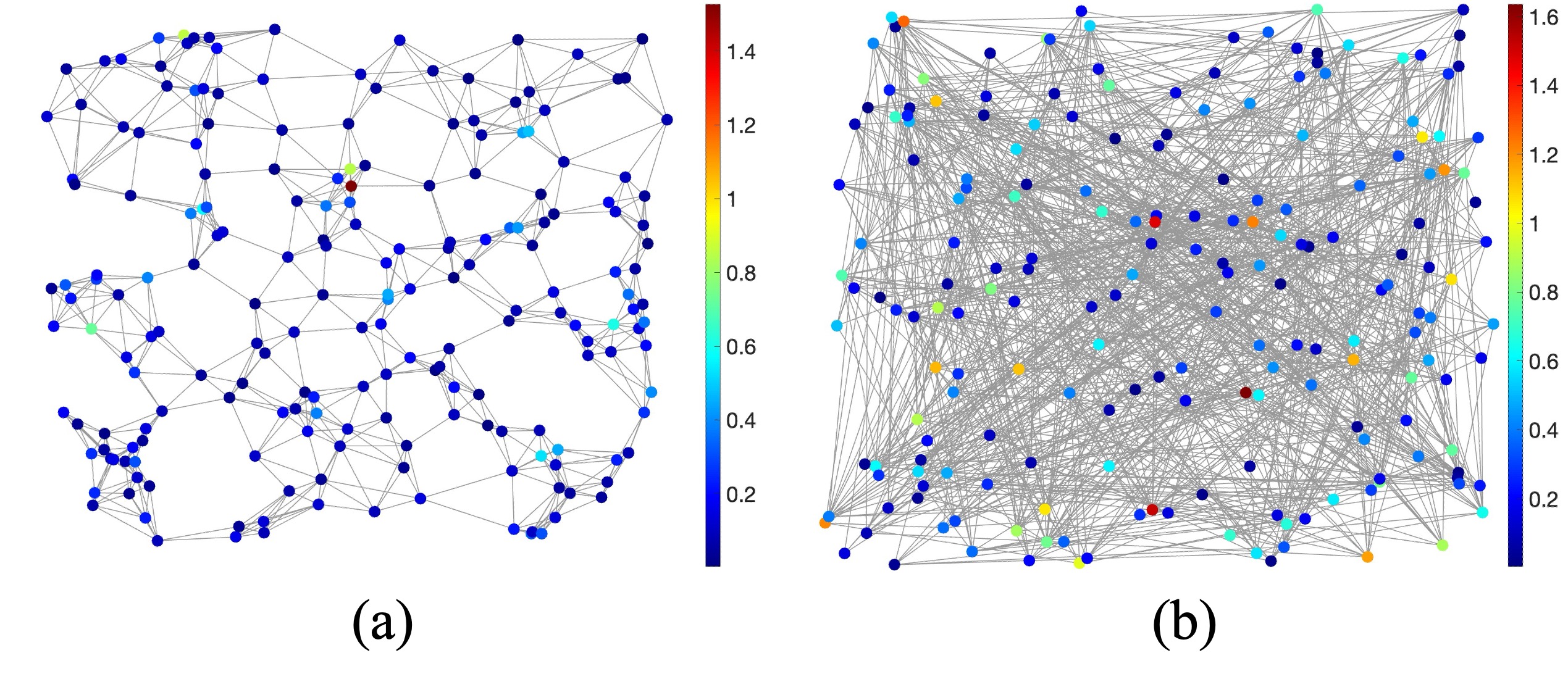}
	\vspace*{-15pt}
	\caption{200-node random graph signals: (a) sensor graph; (b) Erd\H{o}s--R\'enyi graph.}
	\label{Randomsignal}
\end{figure}
Seven sampling strategies are considered for comparison, including six greedy methods and one random sampling baseline. In these experiments, the fractional order is fixed at $\alpha = 0.7$, and the MaxCut method adopts the parameter $k=6$ based on prior empirical observations. For each sampling strategy, the sampling set $\mathcal{S}$ is selected with its cardinality varying from 0 to 200. Given a selected sampling set, the original graph signal is reconstructed using a GFRFT-based reconstruction operator. The reconstruction performance is evaluated in terms of the mean squared error (MSE) between the original signal and the reconstructed signal. The reconstruction results on the sensor graph and the Erd\H{o}s--R\'enyi graph are reported in Fig.~\ref{Randomresult}.
\begin{figure}[h!]
	\centering
	\includegraphics[width=0.9\linewidth]{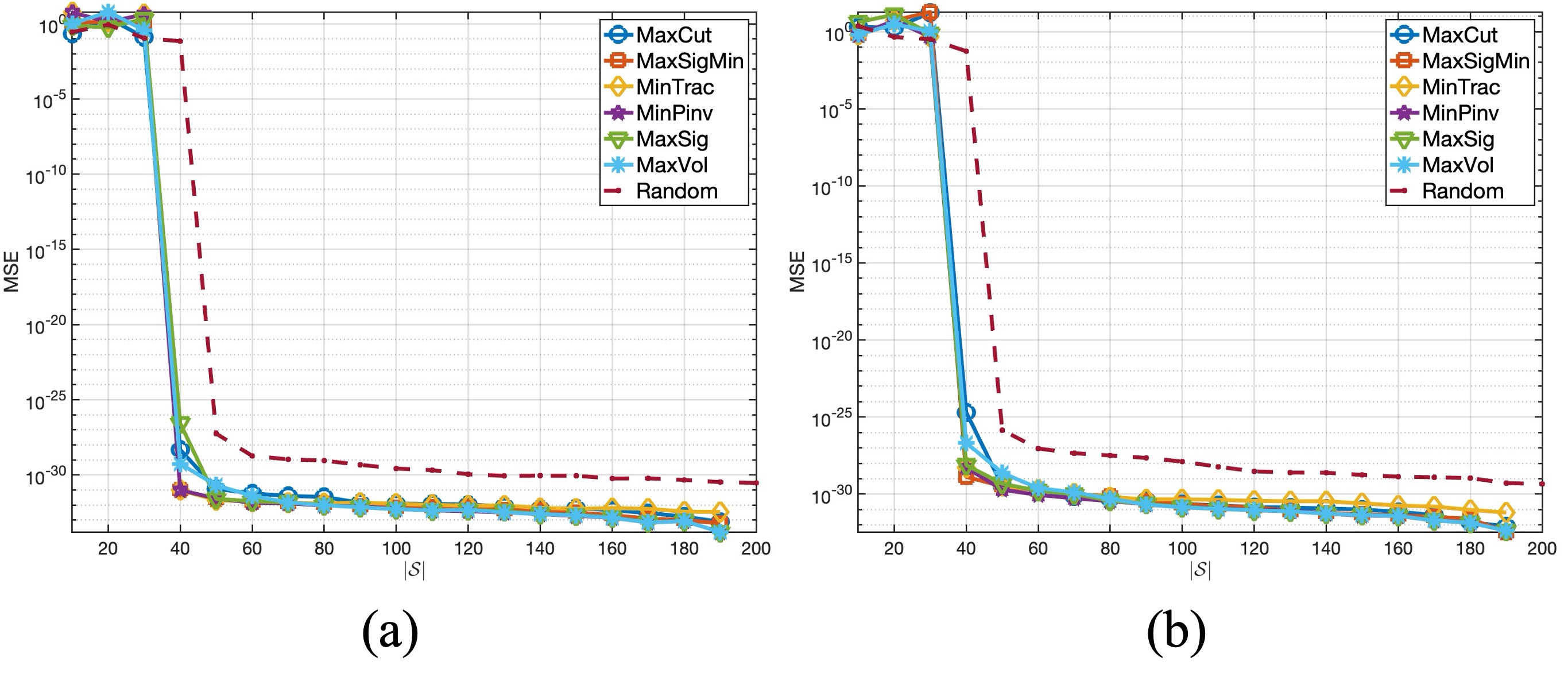}
	\vspace*{-15pt}
	\caption{MSE versus number of samples for different sampling strategies on (a) random sensor graph and (b) Erd\H{o}s--R\'enyi graph.}
	\label{Randomresult}
\end{figure}

As expected, for all sampling methods, the reconstruction error decreases monotonically as the number of sampled nodes increases once $|\mathcal{S}| \geq |\mathcal{F}|$, which confirms the validity of the proposed GFRFT-based sampling and reconstruction framework. Moreover, the greedy sampling strategies consistently achieve significantly lower reconstruction errors than random sampling. 
To further investigate the influence of the fractional order, the sampling size is fixed at $|\mathcal{S}| = 40$, and the reconstruction signal-to-noise ratio (SNR) is evaluated for different values of $\alpha$. The corresponding results for all sampling methods are summarized in Table~\ref{tab:alpha_snr}. 
\begin{table*}[t]
	\centering
	\tiny
	\caption{SNR (dB) comparison of different GFRFT-based sampling methods under varying fractional orders $\alpha$}
	\label{tab:alpha_snr}
	\setlength{\tabcolsep}{5pt}
	\renewcommand{\arraystretch}{1.15}
	\begin{tabular}{c|rrrrrrrrrrrrr}
		\hline
		\hline
		& 0.40 & 0.45 & 0.50 & 0.55 & 0.60 & 0.65 & 0.70 & 0.75 & 0.80 & 0.85 & 0.90 & 0.95 & 1.00 \\
		\hline
		\multicolumn{14}{c}{Random sensor graph} \\
		\hline
		MaxCut
		& 3.62 & -16.57 & -11.90 & -11.01& 6.66 & 13.08 & \textbf{270.49} & 10.77 & 2.63 & -2.63 & -7.92 & -15.13 & -55.06 \\
		
		MaxSigMin
		& 3.62 & 4.41 & 5.48 & 6.50 & 11.13 & 17.03 & \textbf{297.29} & 18.28 & 12.76 & 9.70 & 7.62 & 6.05 & 4.78 \\
		
		MinTrac
		& 3.62 & 4.41 & 5.48 & 6.25 & 9.97 & 17.74 & \textbf{297.34} & 18.25 & 12.76 & 9.70 & 7.62 & 6.05 & 4.78 \\
		
		MinPinv
		& 3.62 & 4.41 & 5.48 & 6.25 & 9.97 & 17.74 & \textbf{297.34} & 18.25 & 12.76 & 9.70 & 7.62 & 6.05 & 4.78 \\
		
		MaxSig
		& 3.62 & 4.41 & 5.48 & -20.15 & -8.78 & -6.87 & \textbf{252.95} & -8.92 & -10.29 & -11.83 & -20.60 & -28.10 & -130.55 \\
		
		MaxVol
		& -1.02 & 3.21 & 2.37 & 6.14 & 9.48 & 11.36 & \textbf{279.83} & 9.28 & -1.33 & -3.46 & -14.83 & -22.62 & -137.39 \\
		
		Random
		& -13.65 & -22.01 & -3.55 & -2.49 & -2.38 & -10.12 & \textbf{-1.18} & -1.38 & -7.06 & -5.16 & -18.70 & -33.21 & -111.71 \\
		
		\hline
		\multicolumn{14}{c}{Erd\H{o}s--R\'enyi graph} \\
		\hline
		MaxCut
		& 4.79 & 5.53 & 6.49 & 4.34 & -3.03 & 9.15 & \textbf{241.09} & 5.64 & 14.18 & 10.98 & 6.66 & 5.31 & 4.26 \\
		
		MaxSigMin
		& 4.79 & 5.53 & 6.49 & 7.85 & 10.27 & 19.30 & \textbf{282.95}& 19.72 & 13.36 & 11.00 & 8.92 & 6.97 & 5.49 \\
		
		MinTrac
		& 4.79 & 5.53 & 6.49 & 6.94 & 12.79 & 19.12 & \textbf{276.95} & 20.03 & 14.30 & 11.15 & 8.75 & 6.97 & 5.49 \\
		
		MinPinv
		& 4.79 & 5.53 & 6.49 & 6.94 & 12.79 & 19.12 & \textbf{276.95} & 20.03 & 14.30 & 11.15 & 8.75 & 6.97 &  5.49 \\
		
		MaxSig
		& 4.79 & 5.53 & 6.49 & -4.66& -3.25& 6.17 & \textbf{275.30} & 19.69 & 13.91 & 10.57 & 8.41 & 6.63 & 5.17 \\
		
		MaxVol
		& 1.64 & 3.82 & 2.94 & 2.09 & 6.21 & 5.74 & \textbf{260.84} & 4.36 & -4.22 & -3.44 & -8.80 & -16.63 &  -19.65 \\
		
		Random
		& -32.04 & -13.47 & -5.87 & -13.76 & -7.92 & -2.25 & \textbf{6.84} & -17.16 & -15.27 & -16.39 & -27.54 & -25.68 & -65.12 \\
		
		\hline
		\hline
	\end{tabular}
\end{table*}
It is observed that the best reconstruction performance is achieved when $\alpha = 0.7$ across all methods. Notably, when $\alpha = 1$, the GFRFT reduces to the GFT, and the corresponding sampling strategies coincide with GFT-based sampling. This comparison demonstrates that GFRFT-based sampling provides greater flexibility and improved robustness compared with classical GFT-based approaches. Motivated by these observations, the next subsection investigates both the computational efficiency and the reconstruction accuracy of classical sampling methods and fast sampling schemes.

\subsection{Comparison of GFRFT Sampling Set Selection}
\begin{figure*}[h!]
	\centering
	\includegraphics[width=\linewidth]{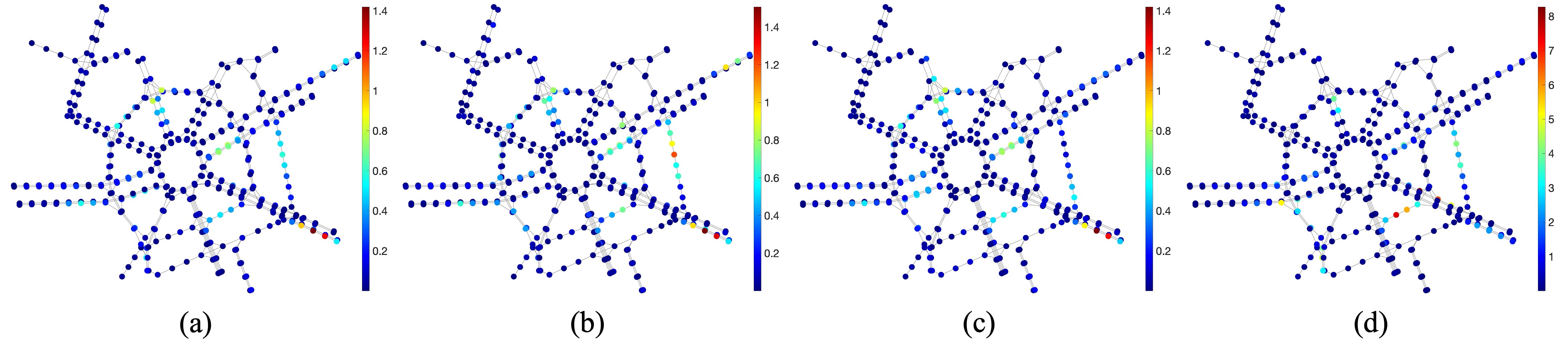}
	\vspace*{-15pt}
	\caption{Reconstruction results of the Rome vehicular traffic data: (a) original signal; (b) MaxSigMin; (c) MaxCov; (d) random sampling.}
	\label{romerecovery}
\end{figure*}
We further investigate the sampling set selection performance of different GFRFT-based methods, with a particular focus on computational efficiency and reconstruction accuracy. In addition to six classical greedy sampling strategies, a fast sampling set selection scheme, namely MaxCov, is included for comparison. For a fixed fractional order $\alpha = 0.9$, bandwidth $|\mathcal{F}| = 20$, and sampling $|\mathcal{S}| = 60$, a GFRFT bandlimited graph signal is constructed from real-world vehicular traffic data $\bm{x}$ defined on the Rome graph\footnote{https://colab.research.google.com/drive/1afJYz0iMJtKnTWUJ0goZi85x3YfJxGvd}. Additive Gaussian noise is introduced to simulate practical measurement perturbations. Specifically, the observed signal is given by $\bm{f} = \mathbf{B}^{\alpha}\,\bm{x} + \mathbf{F}^{-\alpha}\bm{\xi}$, where $\bm{\xi}$ is an i.i.d. noise vector following the distribution $\mathcal{N}(0, 5\times10^{-3})$. For each sampling strategy, a sampling set is first selected, after which signal reconstruction is performed using the corresponding GFRFT reconstruction operator.

The reconstruction results are illustrated in Fig.~\ref{romerecovery}. For the original signal shown in panel (a), among the six classical sampling strategies, the MaxSigMin achieves the best performance with an MSE of 0.04, as shown in (b). In contrast, the proposed MaxCov method further reduces the reconstruction error to 0.01, as depicted in (c), while random sampling yields a significantly larger error of 1.38, as shown in (d). Both quantitative and visual results demonstrate that the fast MaxCov algorithm achieves superior reconstruction performance. 
In addition, the distributions of the selected sampling nodes are presented in Fig.~\ref{romenodes}. It can be observed that MaxCut, MaxSigMin, MinTrac, MinPinv, MaxVol, and the fast MaxCov method all tend to select sampling nodes that are well distributed over the graph, thereby ensuring broad spatial coverage. Notably, MinPinv and MaxSig produce highly similar sampling patterns, which reflects their optimization objectives related to the minimization or maximization of singular values. Random sampling is included solely as a baseline for comparison.
\begin{figure*}[h!]
	\centering
	\includegraphics[width=\linewidth]{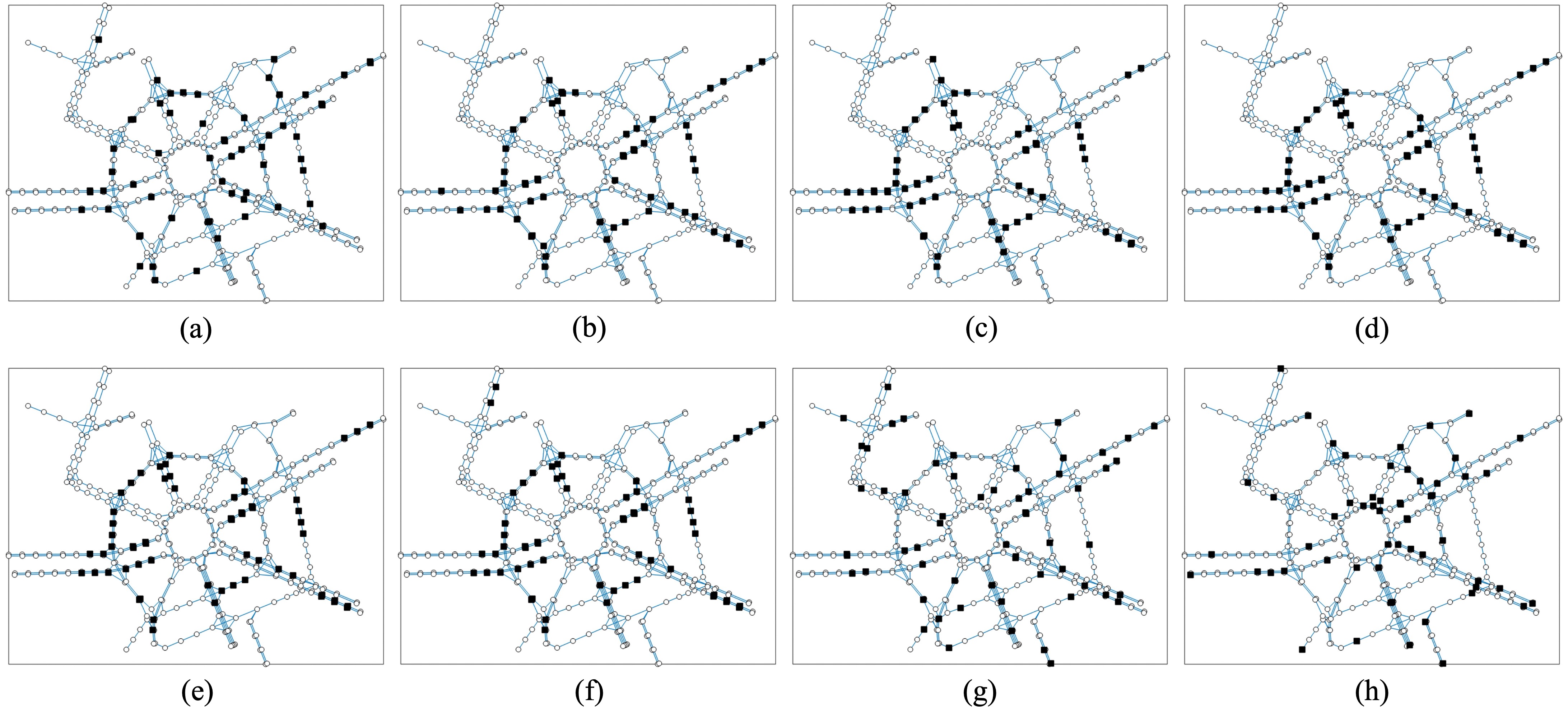}
	\vspace*{-15pt}
	\caption{Locations of the sampled nodes: (a) MaxCut; (b) MaxSigMin; (c) MinTrac; (d) MinPinv; (e) MaxSig; (f) MaxVol; (g) MaxCov; (h) random sampling, shown in black.}
	\label{romenodes}
\end{figure*}

\subsection{Comparison of GFRFT Reconstruction Time}
To evaluate the computational scalability of different sampling strategies, we compare their reconstruction time as a function of the graph size. The experiments are conducted on community graphs with the number of vertices ranging from $200$ to $2000$. For each graph, the sampling size is set to $|\mathcal{S}| = N/10$, while $\alpha = 0.5$ and $|\mathcal{F}| = 20$. A bandlimited graph signal is generated in the GFRFT domain as $\bm{f} = \mathbf{F}^{-\alpha} \left( \bm{z}_{\mathcal{F}} + \bm{\xi} \right)$, where $\bm{z}_{\mathcal{F}} \in \mathbb{R}^{N}$ is a random vector supported on the fractional frequency set $\mathcal{F}$, whose entries indexed by $\mathcal{F}$ are drawn from $\mathcal{N}(0, 0.1)$ and are zero elsewhere, and
$\bm{\xi}$ following $\mathcal{N}(0, 0.01)$. All sampling strategies are applied to the same signal realization for a fair comparison. The reported runtime corresponds to the sampling set selection stage, which dominates the overall computational cost in practical scenarios. Signal reconstruction is then performed using the corresponding GFRFT-based operators to verify the effectiveness of the selected sampling sets.

The runtime comparison is shown in Fig.~\ref{time}. As expected, the computational cost of all greedy sampling methods increases with the graph size. In particular, strategies such as MaxCut, MinTrac, and MaxSig exhibit a rapid growth in runtime as $N$ increases, reflecting the high complexity induced by iterative updates and repeated matrix operations. In contrast, excluding the random baseline, the MaxCov method consistently achieves significantly lower runtime across all graph sizes and demonstrates a much more moderate growth trend with respect to $N$. 
\begin{figure}[h!]
	\centering
	\includegraphics[width=0.7\linewidth]{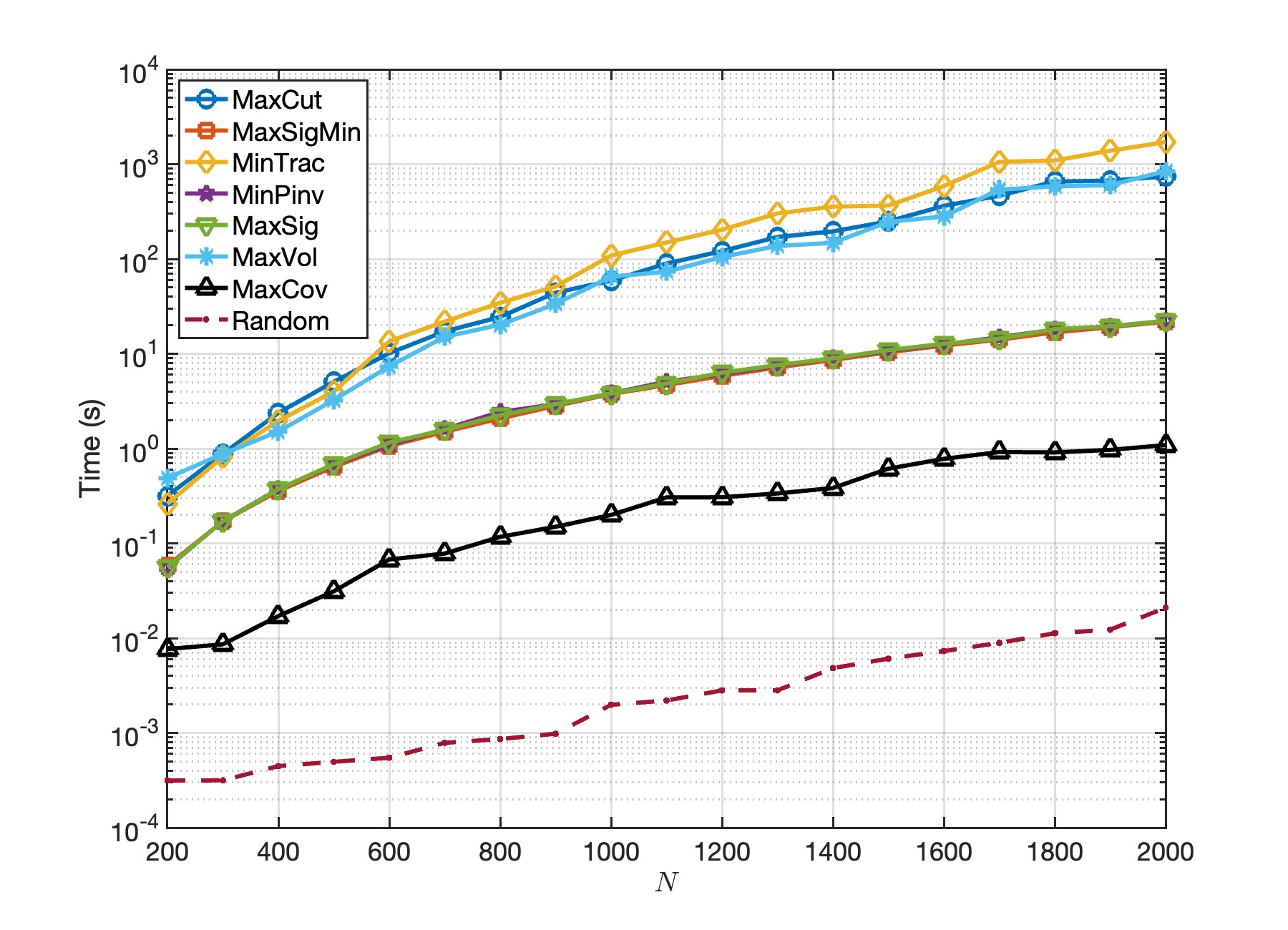}
	\vspace*{-15pt}
	\caption{Runtime comparison of different GFRFT-based sampling strategies as a function of the graph size $N$.}
	\label{time}
\end{figure}
To ensure that the observed computational gains are not achieved at the expense of reconstruction accuracy, Table~\ref{mse} reports the corresponding MSE for methods. 
\begin{table}[h!]
	\footnotesize
	\centering
	\caption{SNR (dB) comparison of different GFRFT-based sampling methods under varying the vertex size $N$}
	\label{mse}
	\renewcommand{\arraystretch}{1.2}
	\setlength{\tabcolsep}{4pt}
	\begin{tabular}{c |r r r r r r }
		\hline
		& 500 & 800 & 1100 & 1400 & 1700 & 2000 \\
		\hline
		MaxCut     & -3.40 & -0.60 & -3.80 & -1.07 & -0.70 & -1.87 \\
		MaxSigMin  & -21.44 & -29.97 & -19.85 & -26.58 & -33.31 & -33.43 \\
		MinTrac   & -7.72 & -13.91 & -28.07 & -23.29 & -23.44 & -32.30 \\
		MinPinv   & -1.92 & -29.17 & -33.80 & -11.76 & -33.64 & -30.70 \\
		MaxSig    & -30.30 & -21.79 & -15.18 & -19.89 & -37.61 & -29.88 \\
		MaxVol    & -15.58 & -25.89 & -40.30 & -31.25 & -25.40 & -25.70 \\
		MaxCov    & \textbf{0.75} & \textbf{0.81} & \textbf{0.56} & \textbf{0.98} & \textbf{0.56} & \textbf{0.70} \\
		Random    & -314.36 & -314.32 & -304.35 & -301.05 & -326.75 & -333.01 \\
		\hline
	\end{tabular}
\end{table}
The results show that MaxCov not only delivers stable and competitive reconstruction performance compared to classical greedy strategies, but also consistently outperforms random sampling. These findings indicate that, although traditional greedy methods remain effective in terms of reconstruction quality, their high computational cost significantly limits their scalability. By contrast, MaxCov achieves a favorable trade-off between computational efficiency and reconstruction accuracy, making it particularly well suited for large-scale GFRFT-based graph signal processing applications.

\subsection{GFRFT Sampling of Non-Bandlimited Sea Clutter Signals}
To further evaluate the performance of the proposed GFRFT-based sampling strategies on real-world data, experiments are conducted on measured sea clutter signals\footnote{http://soma.mcmaster.ca/ipix/dartmouth/cdf001\_050.html.}. 
The raw in-phase and quadrature radar measurements are first extracted, followed by data correction to eliminate possible overflow effects. 
An automatic preprocessing and normalization procedure is then applied to compensate for amplitude imbalance and normalize the radar returns.

In the experiments, a signal of dimension $N=200$ corresponding to a fixed range bin is considered. 
A graph is constructed using a Gaussian kernel-based method with parameter $\sigma=1$ and edge connection probability set to $0.05$, where each vertex represents a time sample and the edge weights are determined by a similarity kernel applied to the radar signal snapshots.

Based on the constructed graph, the GFRFT is applied with an effective bandwidth $|\mathcal{F}|=20$. 
All GFRFT-based sampling strategies are investigated. 
For each strategy, the fractional order $\alpha$ is first coarsely explored and then refined around promising values. 
With a fixed sampling number $|\mathcal{S}|=50$, the corresponding sampling set is selected and the original non-bandlimited sea clutter signal is reconstructed using the proposed fractional reconstruction framework. 
For each method, the optimal fractional order minimizing the reconstruction error is identified.

Table~\ref{snr_alpha} reports the corresponding maximum SNR values for quantitative comparison.  Among all considered strategies, MaxCov method consistently achieves the best reconstruction performance and attains the highest SNR, with its optimal value obtained at $\alpha = 0.9$.
\begin{table}[t]
	\centering
	\footnotesize
	\caption{Optimal $\alpha$ and corresponding maximum SNR for different GFRFT-based sampling methods.}
	\label{snr_alpha}
	\begin{tabular}{l c r}
		\toprule
		Method & Best $\alpha$ & Maximum SNR (dB) \\
		\midrule
		MaxCut    & 0.10 & 0.25 \\
		MaxSigMin & 0.80 & -20.58 \\
		MinTrac   & 1.30 & -11.30 \\
		MinPinv   & 0.60 & -12.57 \\
		MaxSig    & 1.50 & -11.85 \\
		MaxVol    & 0.70 & -10.34 \\
		MaxCov    & 0.90 & \textbf{1.01} \\
		\bottomrule
	\end{tabular}
\end{table}

To further assess the statistical reconstruction behavior, the empirical cumulative distribution functions (CDF) of the reconstructed signal magnitudes are analyzed and compared with that of the original sea clutter signal.
As shown in Fig.~\ref{CDF}, the proposed GFRFT-based sampling methods, particularly MaxCov, closely match the distribution of the original signal, demonstrating their ability to preserve both global and local statistical characteristics of real-world sea clutter data.
\begin{figure}[h!]
	\centering
	\includegraphics[width=0.7\linewidth]{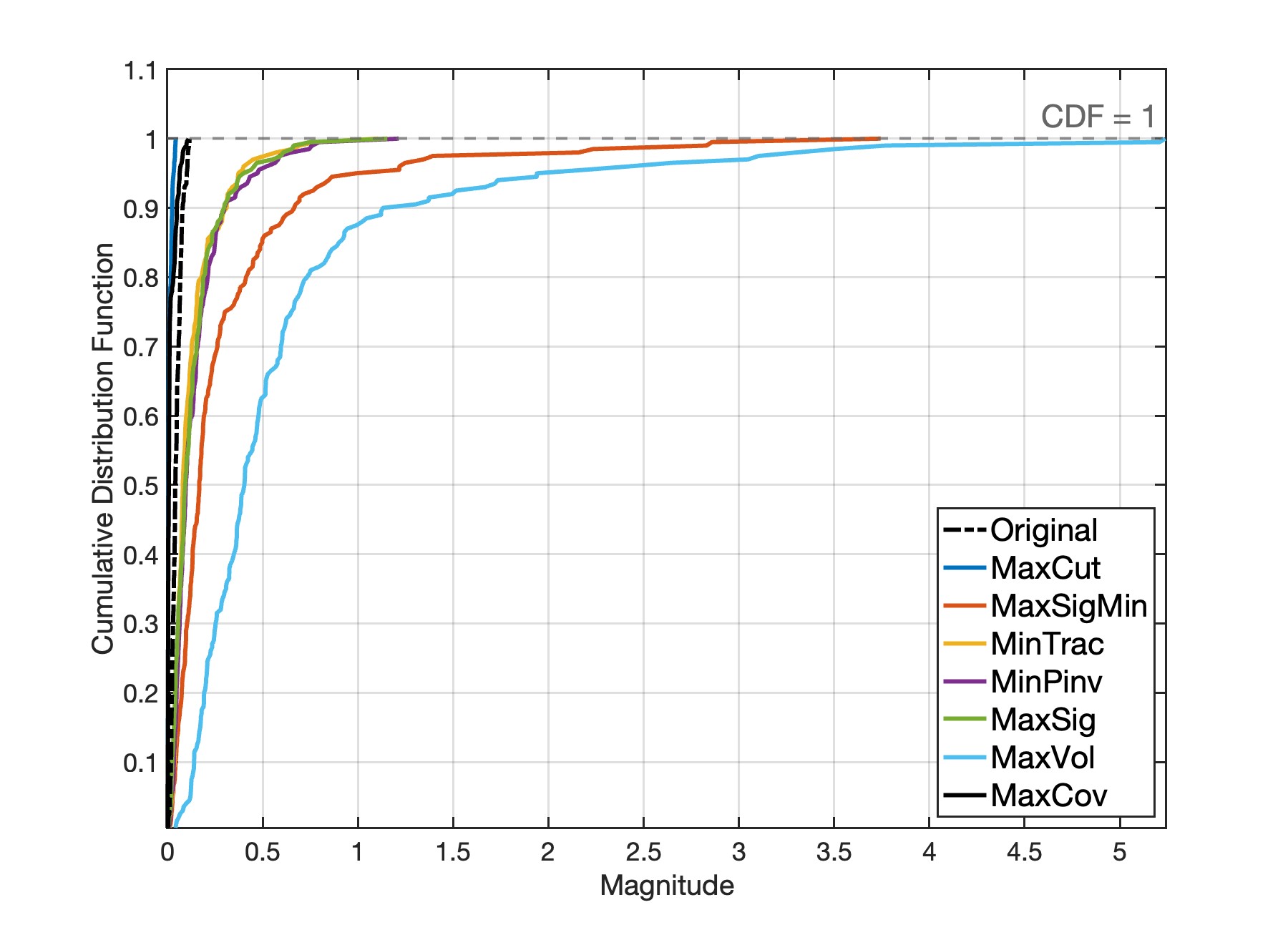}
	\vspace*{-15pt}
	\caption{Comparison of statistical characteristics of various methods.}
	\label{CDF}
\end{figure}

\section{Conclusion}
\label{Conclusion}
This paper presents a unified and efficient GFRFT sampling framework that incorporates the definition of bandlimited graph signals, the associated sampling theorem, perfect reconstruction conditions, and localization operators. Sampling strategies are developed under criteria including maximum cutoff frequency, minimum reconstruction error, and maximum localized basis, with explicit representations of their localization operators. Leveraging these operators, a computationally efficient sampling set selection method is introduced. Numerical simulations and application experiments demonstrate the effectiveness and advantages of the proposed strategies in reconstruction accuracy and computational efficiency. Future work will explore extensions to complex-valued graph signals, including applications in quantum networks and other radar signal processing.

\appendix
\section{Proof of Theorem \ref{thm1}}
\label{AA}

First, assume that $\bm{x}$ is perfectly localized in both the vertex and fractional spectral domains. Then $\mathbf{B}^{\alpha}\bm{x} = \bm{x}$, and $\mathbf{D}\bm{x} = \bm{x}$.
Applying these identities repeatedly gives
\[
\mathbf{B}^{\alpha}\mathbf{D}\mathbf{B}^{\alpha}\bm{x} 
= \mathbf{B}^{\alpha}\mathbf{D}\bm{x} 
= \mathbf{B}^{\alpha}\bm{x} 
= \bm{x},
\]
so that $\bm{x}$ is an eigenvector of $\mathbf{B}^{\alpha}\mathbf{D}$ with eigenvalue $1$.  
Since $\mathbf{B}^{\alpha}$ and $\mathbf{D}$ are orthogonal projectors, $\|\mathbf{B}^{\alpha}\mathbf{D}\|_2 \le 1$, hence $\|\mathbf{B}^{\alpha}\mathbf{D}\|_2 = 1$. Noting that $(\mathbf{B}^{\alpha}\mathbf{D})^{\mathrm{H}} = \mathbf{D}\mathbf{B}^{\alpha}$, and a matrix shares singular values with its Hermitian conjugate, we have
\[
\|\mathbf{D}\mathbf{B}^{\alpha}\|_2 = \|\mathbf{B}^{\alpha}\mathbf{D}\|_2 = 1.
\]

Conversely, assume $\|\mathbf{B}^{\alpha}\mathbf{D}\|_2 = 1$. Then there exists $\bm{x} \neq 0$ such that 
\[
\|\mathbf{B}^{\alpha}\mathbf{D}\bm{x}\|_2 = \|\bm{x}\|_2.
\]
Equality for orthogonal projectors holds only if $\mathbf{B}^{\alpha}\mathbf{D}\bm{x} = \bm{x}$. Left-multiplying by $\mathbf{B}^{\alpha}$ and using $(\mathbf{B}^{\alpha})^2 = \mathbf{B}^{\alpha}$ yields $\mathbf{B}^{\alpha}\mathbf{D}\bm{x}
= \mathbf{B}^{\alpha}\bm{x}$, combining two relations leads to $\mathbf{B}^{\alpha}\bm{x} = \bm{x}$. Substituting back gives $\mathbf{D}\bm{x} = \bm{x}$, proving perfect localization in both domains.

Finally, since $(\mathbf{B}^{\alpha}\mathbf{D})^{\mathrm{H}} = \mathbf{D}\mathbf{B}^{\alpha}$, and thus $\|\mathbf{B}^{\alpha}\mathbf{D}\|_2=1$ is equivalent to $\|\mathbf{D}\mathbf{B}^{\alpha}\|_2=1$.
This completes the proof.

\section{Proof of Theorem \ref{thm2}}
\label{AB}

Let $\mathbf{D}=\mathbf{I}-\overline{\mathbf{D}}$ and $\bm{x}\in\mathcal{B}^{\alpha}$, so that $\mathbf{B}^{\alpha}\bm{x}=\bm{x}$.  
For a reconstruction operator $\mathbf{R}$, the reconstruction error is
\[
\bm{x}-\mathbf{R}\mathbf{D}\bm{x}
= \bm{x}-\mathbf{R}(\mathbf{I}-\overline{\mathbf{D}})\bm{x}
= \bm{x}-\mathbf{R}(\mathbf{I}-\overline{\mathbf{D}}\mathbf{B}^{\alpha})\bm{x}. 
\]
Hence any $\mathbf{R}$ satisfying $\mathbf{R}(\mathbf{I}-\overline{\mathbf{D}}\mathbf{B}^{\alpha})=\mathbf{B}^{\alpha}$ ensures $\mathbf{R}\mathbf{D}\bm{x}=\bm{x}$ for all $\bm{x}\in\mathcal{B}^{\alpha}$. If $\|\mathbf{B}^{\alpha}\overline{\mathbf{D}}\|_2<1$, then $\mathbf{I}-\overline{\mathbf{D}}\mathbf{B}^{\alpha}$ is invertible via the Neumann series $\sum_{k=0}^{\infty}(\overline{\mathbf{D}}\mathbf{B}^{\alpha})^k$, and one can choose
\[
\mathbf{R} = \mathbf{B}^{\alpha} (\mathbf{I}-\overline{\mathbf{D}}\mathbf{B}^{\alpha})^{-1},
\]
which guarantees $\mathbf{R}\mathbf{D}\bm{x}=\bm{x}$ for all $\bm{x}\in\mathcal{B}^{\alpha}$.

Conversely, if $\|\mathbf{B}^{\alpha}\overline{\mathbf{D}}\|_2=1$, there exists $\bm{v}\neq 0$ with $\|\mathbf{B}^{\alpha}\overline{\mathbf{D}}\bm{v}\|_2=\|\bm{v}\|_2$. Let $\bm{w}=\mathbf{B}^{\alpha}\bm{v}\in\mathcal{B}^{\alpha}$. Then $\mathbf{B}^{\alpha}\overline{\mathbf{D}}\bm{w}=\bm{w}$, but $\overline{\mathbf{D}}$ annihilates components on $\mathcal{S}$, so $\bm{w}$ is supported entirely on $\overline{\mathcal{S}}$ and $\mathbf{D}\bm{w}=\mathbf{0}$. Hence sampling yields zero and no reconstruction operator can recover $\bm{w}\neq 0$, making perfect recovery for all $\bm{x}\in\mathcal{B}^{\alpha}$ impossible.

\section{Proof of Theorem \ref{thm3}}
\label{AC}
Thus the reconstruction formula $\bm{x}_{\mathcal{R}}
= \mathbf{T}^{\alpha}_{\mathcal{V}\mathcal{S}} \left( \mathbf{T}^{\alpha}_{\mathcal{S}} \right)^{\dag} \bm{x}_{\mathcal{S}}$, and substitute 
\(
\mathbf{T}^{\alpha}= \mathbf{F}^{-\alpha} h(\mathbf{\Delta}^{\alpha}) \mathbf{F}^{\alpha}.
\)
This gives
\[
\begin{aligned}
	\bm{x}_{\mathcal{R}}
	=& \left( \mathbf{F}^{-\alpha} h\left( \mathbf{\Delta}^{\alpha} \right) \mathbf{F}^{\alpha} \right)_{\mathcal{V}\mathcal{S}} \left( \left( \mathbf{F}^{-\alpha} h\left( \mathbf{\Delta}^{\alpha} \right) \mathbf{F}^{\alpha} \right)_{\mathcal{S}} \right)^{\dag} \bm{x}_{\mathcal{S}} \\
	=& \mathbf{F}^{-\alpha} h\left( \mathbf{\Delta}^{\alpha} \right) \mathbf{F}^{\alpha}_{\mathcal{S} \mathcal{V}} \left( \mathbf{F}^{-\alpha}_{\mathcal{S} \mathcal{V}} h\left( \mathbf{\Delta}^{\alpha} \right) \mathbf{F}^{\alpha}_{\mathcal{S} \mathcal{V}} \right)^{\dag} \bm{x}_{\mathcal{S}} \\
	=& \mathbf{F}^{-\alpha} h^{1/2} \left( \mathbf{\Delta}^{\alpha} \right) \left( \mathbf{F}^{-\alpha}_{\mathcal{S} \mathcal{V}} h^{1/2} \left( \mathbf{\Delta}^{\alpha} \right) \right)^{\dag} \bm{x}_{\mathcal{S}}.
\end{aligned}
\]

To normalize the filter response, rewrite the expression as
\begin{equation}
	\bm{x}_{\mathcal{R}}
	= \mathbf{F}^{-\alpha}
	\left( \frac{h(\mathbf{\Delta}^{\alpha})}{\rho} \right)^{1/2}
	\left(
	\mathbf{F}^{-\alpha}_{\mathcal{S}\mathcal{V}}
	\left( \frac{h(\mathbf{\Delta}^{\alpha})}{\rho} \right)^{1/2}
	\right)^{\dag}
	\bm{x}_{\mathcal{S}},
	\label{xR}
\end{equation}
where 
\(
\rho = \min_{0 \le i \le |\mathcal{F}|-1} [h(\mathbf{\Delta}^{\alpha})]_{ii}.
\)
For \( i \ge |\mathcal{F}| \), the corresponding entries of  
\(
(h(\mathbf{\Delta}^{\alpha})/\rho)^{1/2}
\)
become negligible. Under this assumption,
\begin{equation}
	\left(
	\mathbf{F}^{-\alpha}_{\mathcal{S}\mathcal{V}}
	\left( \frac{h(\mathbf{\Delta}^{\alpha})}{\rho} \right)^{1/2}
	\right)^{\dag}
	\approx
	\left(
	\frac{h(\mathbf{\Delta}^{\alpha}_{\mathcal{F}})}{\rho}
	\right)^{-1/2}
	\left( \mathbf{F}^{-\alpha}_{\mathcal{S}\mathcal{F}} \right)^{\dag}.
	\label{approximation}
\end{equation}

Substituting \eqref{xR} into \eqref{approximation} yields
\[
\bm{x}_{\mathcal{R}} 
= \mathbf{F}^{-\alpha}_{\mathcal{V} \mathcal{F}} \left( \mathbf{F}^{-\alpha}_{\mathcal{S} \mathcal{F}} \right)^{\dag} \bm{x}_{\mathcal{S}} = \mathbf{D}_{\mathcal{S}} \mathbf{F}^{-\alpha}_{\mathcal{V}} \mathbf{\Sigma}_{\mathcal{F}} \mathbf{F}^{\alpha}_{\mathcal{V}} \bm{x}_{\mathcal{S}}= \mathbf{R} \bm{x}_{\mathcal{S}},
\]
which completes the proof.

\section*{Declaration of competing interest}
The authors declare that they have no known competing financial interests or
personal relationships that could have appeared to influence the work reported
in this paper.

\section*{Acknowledgments}
This work were supported by the Natural Science Foundation of Beijing Municipality [No. 4242011], and the National Natural Science Foundation of China [No. 62571042].


\

\

\


\end{document}